\documentclass[11pt,a4paper]{amsart}
\usepackage[utf8]{inputenc}
\usepackage{amsmath}
\usepackage{amsfonts}
\usepackage{amssymb}
\usepackage{amsthm}
\usepackage{setspace}
\usepackage{hyperref}
\usepackage{graphicx}
\usepackage{stmaryrd}
\usepackage[shortlabels]{enumitem}
\usepackage{mathtools}
\usepackage{subcaption}
\usepackage[T1]{fontenc}
\usepackage{xcolor}
\usepackage{float}
\usepackage{tikz-cd}
\usepackage{tikz}
\usepackage[noadjust]{cite}
\usepackage{cite}
\usepackage{ragged2e}
\usepackage[margin=1.2in]{geometry}
\usepackage{cleveref}
\graphicspath{ {./images/} }
\newtheorem{theorem}{Theorem}[section]
\newtheorem{proposition}[theorem]{Proposition}
\newtheorem{lemma}[theorem]{Lemma}
\newtheorem*{mtheorem}{Main Theorem}
\newtheorem*{corollaryA}{Corollary A}
\newtheorem*{corollaryB}{Corollary B}
\newtheorem*{corollaryC}{Corollary C}
\newtheorem*{corollaryD}{Corollary D}

\usepackage{hyperref}

\usepackage[scr=boondox,scrscaled=1.05]{mathalfa}

\newtheorem{corollary}[theorem]{Corollary}

\newtheorem{case}[theorem]{Case}

\theoremstyle{definition}
\newtheorem{claim}[theorem]{Claim}
\newtheorem{definition}[theorem]{Definition}
\newtheorem{question}[theorem]{Question}
\newtheorem{problem}[theorem]{Problem}

\theoremstyle{remark}
\newtheorem{remark}[theorem]{Remark}
\newtheorem{example}[theorem]{Example}

\newcommand*\sg[1]{\{ #1 \}}

 \newcommand{\Lk}{\mathrm{Lk}}
 \newcommand{\Th}{\mathrm{Th}}

\graphicspath{{pictures/}}

\title{Fit systolic groups, exactly}

\author{Martín Blufstein}
\address[M.~Blufstein]{Departamento de Matemática-IMAS (CONICET), FCEyN, Universidad de Buenos Aires, Buenos Aires, Argentina}
\email{mblufstein@dm.uba.ar}

\author{Victor Chepoi}
\address[V.~Chepoi]{Aix-Marseille Universit\'e and CNRS, LIS, Marseille, France}
\email{victor.chepoi@lis-lab.fr}

\author{Huaitao Gui}
\address[H.~Gui]{Institut for Matematiske Fag, University of Copenhagen, 2100 Copenhagen, Denmark}
\email{hg@math.ku.dk}

\author{Damian Osajda}
\address[D.~Osajda]{Institut for Matematiske Fag, University of Copenhagen, 2100 Copenhagen, Denmark}
\address{Instytut Matematyczny,
	Uniwersytet Wroc\l awski\\
	pl.\ Grun\-wal\-dzki 2/4,
	50--384 Wroc\-{\l}aw, Poland}
\email{dosaj@math.uni.wroc.pl}

\begin{document}

\begin{abstract}
A systolic complex/bridged graph is \emph{fit} when its (metric) intervals are ``not too large''. We prove that uniformly locally finite fit systolic complexes have Yu's Property A. In particular, groups acting properly on such complexes have Property A, (equivalently) they
		are exact, 
        and (equivalently) they are boundary amenable.
		As applications we show that groups from a class containing all large-type Artin groups, as well as all finitely presented graphical $C(3)$--$T(6)$ small cancellation groups, and finitely presented classical $C(6)$ small cancellation groups are exact. We also
provide further examples.
Our proof relies on a combinatorial criterion for Property~A
due to \v{S}pakula and Wright.
\end{abstract}

\maketitle

\section{Introduction}

    \subsection{Main results}
	Embedding, even just in a coarse way, a metric space into a Hilbert space allows one to draw strong conclusions about the former space using the geometry of the latter one. This applies also to (countable) groups equipped with a suitable metric.	In particular, the coarse Baum-Connes conjecture holds for spaces embeddable coarsely into a Hilbert space \cite{Yu2000}. In the case of countable groups the coarse Baum-Connes conjecture implies e.g.\ the strong Novikov conjecture \cite{Roe96,Hig2000,STY2002}.  
	
	Guoliang Yu \cite{Yu2000} introduced a coarse geometric property---\emph{Property A}---as a criterion for coarse embeddability into a Hilbert space. This property can be seen as a ``non-equivariant'', ``coarse'', or ``weak'' variant of amenability. And, similarly as for amenability, there exist numerous equivalent definitions of various flavours---see Definition~\ref{d:A} for one. Not surprisingly, Yu's Property A became of great interest on its own and has been intensively studied over the last decades; see \cite{BrOz,Willett2009,NoYu2023} for surveys. For a group $H$ seen as a metric space when equipped with the word metric with respect to a finite generating set, the following properties are equivalent \cite{HiRo2000,Oz2000,GuKa2002}:
	\begin{itemize}
		  \item $H$ has Yu's  Property A;
		  \item $H$ is \emph{exact}, that is, the functor $C^\ast_r(H, \square)$ is exact;
		  \item $H$ is \emph{boundary amenable}, that is, $G$ admits an amenable action on a nonempty compact space. 
	\end{itemize}
    These are only three different characterizations, from three different perspectives, respectively: geometric, analytic, and dynamical. See Section~\ref{sec:A} for further characterizations of Property A for groups. 
	At the moment there exist only two constructions of finitely generated groups without Yu's Property A: Gromov random monster groups \cite{Gr2003}, and graphical small cancellation groups constructed in \cite{Osa-scl}. Basically, all other groups appearing in the literature are believed to have Property A. However, this has been proven only for a few classes of groups, among them: (Gromov) hyperbolic groups \cite{Ada1994}, 
    Coxeter groups \cite{DraJan}, one relator groups \cite{Gue2002}, linear groups \cite{GHW2005}, relatively hyperbolic groups (with appropriate parabolics) \cite{Oza2006}, CAT(0) cubical groups \cite{CaNi2005,BCGNW,GuNi2011}), mapping class groups of finite-type surfaces \cite{Kida2006,Ham2009}, some isometry groups of buildings \cite{DymSch,Cam2009,Lec2010}, groups acting geometrically on $2$-dimensional systolic complexes \cite{HuOs_systolic}, outer automorphism groups of free groups \cite{BGH2022}, and $2$-dimensional Artin groups of hyperbolic type \cite{HoHu2021}.   
	For example, it is still unknown whether groups acting geometrically on CAT(0) spaces have Property A, or are coarsely embeddable into Hilbert spaces---see some related questions in Section~\ref{sec:final}.

\medskip
	
In this article we establish Property A for a broad class of systolic complexes by introducing the notion of fitness. \emph{Systolic} complexes are flag simplicial complexes whose $1$-skeleta are \emph{bridged} graphs, that is, graphs in which isometric cycles have length $3$---see details in Section~\ref{sec:premgr}. Bridged graphs were first introduced by Soltan and Chepoi \cite{SoCh}, and Farber and Jamison \cite{FaJa}. Systolic complexes were first studied in \cite{Ch_CAT}, and they were rediscovered by Januszkiewicz and \'Swi{\c a}tkowski \cite{JaSw}, and Haglund \cite{Hag2003} within the framework of Geometric Group Theory. A group is \emph{systolic} if it acts \emph{geometrically}, that is, properly and cocompactly, on a systolic complex. Systolicity is a form of combinatorial nonpositive curvature, and providing a systolic structure for a group allows one to grasp a fair control over its geometry. 
The theory has been successfully used for constructing new exotic examples of groups \cite{JS2003,JaSw}, and for equipping ``classical'' groups with a nonpositive curvature structure, thus allowing to prove new results about them \cite{HoOs, OsPry2018, FukOg2020}. The following main result of the current article extends this latter line of research.
We introduce the subclass of \emph{fit systolic complexes}---those where level sets of intervals are quasi-segments (see Section~\ref{sec:fit} for details).

\begin{mtheorem}\label{thm:property_A}
    Uniformly locally finite fit systolic complexes have Yu's Property A.
\end{mtheorem}

Below we present our main applications of the Main Theorem---to Artin groups (Corollary A), and to small cancellation groups (Corollaries B \&~C). We also present a $3$-dimensional pseudomanifold example (Corollary D). Besides them, the theorem can be used to reprove Property A in other cases: for $2$-dimensional systolic complexes (see Remark~\ref{rem:2dim}) and for few smaller classes of complexes. We believe that the Main Theorem, as well as our general approach (see e.g.\ Remark~\ref{rem:generalfit}) will be useful for other large classes of groups and complexes; cf.\ Section~\ref{sec:final}.

Along the way we prove a result on general bridged graphs/systolic complexes that we believe is of independent interest---Proposition~\ref{p:convergence_of_normal_clique_paths}. We also provide a characterization of fit systolic complexes (Proposition~\ref{prop:equivalent_definition}) and a useful criterion for fitness (Corollary~\ref{c:fit_condition}). We use them in our proofs and we believe that they will be useful in future studies.
    
\subsubsection*{Large-type Artin groups} Artin groups are a family of groups generalizing e.g.\ braid groups, and being intensively studied recently, due to their mysterious nature.
Recall, that for a finite simplicial graph $\Gamma$ with every edge $ab$ between vertices $a,b$ labelled by a positive integer $m_{ab}\geq 2$, the \emph{Artin group} $A_\Gamma$ associated to $\Gamma$ is the group with the presentation
\[
\langle a\in V(\Gamma)\ |\ \underbrace{aba\cdots}_{m_{ab}} = \underbrace{bab \cdots}_{m_{ab}}, \text{ whenever } a \text{ and } b \text{ are connected by an edge in } \Gamma\rangle,
\]
An Artin group $A_\Gamma$ is of \emph{large type} if all the edge labels $m_{ab}$ in its defining graph $\Gamma$ are at least $3$, and is (more generally) of \emph{almost large type} if the graph $\Gamma$ contains no triangle with an edge labelled by $2$ and no square with three edges labelled by $2$.

In general, the geometry of arbitrary Artin groups is still poorly understood despite the fact that all Artin groups are believed to satisfy some notion of nonpositive curvature; see e.g.\ \cite{vdL,GuNi2011,HuOs_systolic,HoHu2021, BlVa} for more information on Artin groups.
Huang and Osajda \cite{HuOs_systolic} equipped almost-large-type Artin groups with a structure of systolic groups. This allows one to draw strong consequences about various features of the groups, and allows us to prove the following.
	
    \begin{corollaryA}
         Almost-large-type Artin groups are exact. 
    \end{corollaryA}

The corollary is an immediate consequence of the Main Theorem, and the fact that almost-large-type Artin groups act geometrically on fit systolic complexes---Proposition~\ref{prop:fit_Artin}.

\subsubsection*{$C(3)$--$T(6)$ and $C(6)$ small cancellation groups}
Small cancellation is a powerful classical technique for constructing infinite groups; see e.g.\ \cite{LynSch2001,Wise2003,OsPry2018,Osa-scl} for some background. Small cancellation complexes are $2$-dimensional CW complexes with some combinatorial nonpositive curvature conditions---see definitions in Section~\ref{sec:smallcanc}. 
The graphical version of small cancellation is a far reaching generalization of the classical setting (that is, the one as presented in \cite{LynSch2001,Wise2003}), that allowed recently to construct examples of groups with exotic properties, e.g.\ new nonexact groups \cite{Osa-scl,Os-rf}.
There are three classical small cancellation conditions: the $C(6)$ condition, the $C(4)$--$T(4)$ condition, and the $C(3)$--$T(6)$ condition (see Definition~\ref{def:smallcanc}). In Propositions~\ref{prop:c3t6_is_systolic} \&~\ref{prop:c3t6_path_level_sets} we equip simply connected graphical $C(3)$--$T(6)$ small cancellation complexes with a fit systolic geometry. Together with our Main Theorem, it has the following consequence. 

    \begin{corollaryB}
        Uniformly locally finite simply connected graphical $C(3)$--$T(6)$ small cancellation complexes have Yu's Property A. Consequently, groups with finite graphical $C(3)$--$T(6)$ presentations are exact.
    \end{corollaryB}
    
Furthermore, in Proposition~\ref{prop:fit_Wise} we show that for every simply connected classical $C(6)$ complex $X$, a simplicial complex introduced by Wise $W(X)$---certain systolic complex naturally associated with $X$ \cite{Wise2003,OsPry2018}---is fit. Hence we get the following.

    \begin{corollaryC}
        Uniformly locally finite simply connected classical $C(6)$ small cancellation complexes have Yu's Property A. Consequently, groups with finite $C(6)$ presentations are exact.
    \end{corollaryC}

\subsubsection*{A $3$-dimensional systolic pseudomanifold} All the groups from the three families explored above act geometrically on $2$-dimensional contractible complexes, so are in a way ``$2$-dimensional''. We next present an example of a group acting geometrically on a $3$-dimensional fit  systolic pseudomanifold. Following a construction by \'Swi{\c a}tkowski \cite{Sw}, Wieszaczewski \cite{JW} defined a systolic group  $G(54,54,72,72)$ acting geometrically on a $3$-dimensional systolic pseudomanifold $K(54,54,72,72)$. We show that  this complex is fit.

    \begin{corollaryD}
        The $3$-dimensional systolic pseudomanifold $K(54,54,72,72)$ is fit and hence it has Property A. Consequently, the group $G(54,54,72,72)$ is exact.
    \end{corollaryD}

See Section~\ref{sec:54547272} for the definitions of $K(54,54,72,72)$ and $G(54,54,72,72)$, and for a discussion---Remark~\ref{r:3psmfld}.

	\subsection{Related results}
    \label{sec:relres}
    The coarse Baum-Connes conjecture and hence the strong Novikov conjecture have been known to hold for systolic groups by \cite{OsPr,FukOg2020}. Property A or coarse embeddability into a Hilbert space are unknown for general systolic complexes or groups, cf.\ Section~\ref{sec:final}.

    As for Artin groups, Property A holds for \emph{right-angled} ones (all labels are $2$) \cite{CaNi2005} and, more generally, for \emph{FC-type} Artin groups (cliques in the defining graph $\Gamma$ induce finite Coxeter groups) \cite{GuNi2011}. Both results use the fact that Artin groups in such classes act nicely on CAT(0) cubical complexes. Large-type Artin groups treated in the current article are seen as very different from FC-type Artin groups. More relevant to our results is the proof of Property A for $2$-dimensional Artin groups of hyperbolic type from \cite{HoHu2021}. Their proof is based on the fact that such groups admit nice actions on Gromov hyperbolic complexes. This approach is not applicable for general large-type Artin groups. An important example (not approachable by techniques from \cite{HoHu2021}) is $A_\Gamma$ for $\Gamma$ being a triangle with all edges labelled by $3$. This is a large-type Artin group (hence $2$-dimensional) which is not of hyperbolic type. It should be mentioned that this particular $A_\Gamma$ is commensurable to the mapping class group of the five-punctured sphere, hence its exactness follows from \cite{Kida2006,Ham2009}. However, most large-type Artin groups are not coarsely equivalent to mapping class groups of surfaces. Finally, let us mention that the proof of exactness in \cite{HoHu2021} is used there to obtain measure equivalence rigidity for corresponding groups; cf.\ Section~\ref{sec:final}.

    For finitely presented graphical $C(6)$, $C(4)$--$T(4)$, or $C(3)$--$T(6)$ small cancellation groups the coarse Baum-Connes conjecture and hence the strong Novikov conjecture follow from combining results from \cite{Wise2003,OsPry2018,FukOg2020,CCGHO25,duda2023}. Property A or coarse embeddability into a Hilbert space were not known for them. See Section~\ref{sec:final} for some open questions.

\subsection{Organization}
	
For proving the Main Theorem we use a criterion by \v{S}pakula and Wright \cite{SpWr}---see Proposition~\ref{prop:criterion}. We also follow the general pattern of a proof of Property A for CAT(0) cubical complexes from \cite{SpWr}. To do this we need to establish Proposition~\ref{p:convergence_of_normal_clique_paths} stating that normal clique paths converge quickly to intervals. For satisfying the item (\ref{item:3}) in Proposition~\ref{prop:criterion} we need the intervals in the systolic complex to be ``not too large''. This leads to the definition of ``fit systolic complex'' -- Definition~\ref{def:fit}. 
To justify the definition, we recall the fundamental property of level sets in systolic complexes being quasi-trees in Section~\ref{sec:fit}.
With all this in hand we prove the Main Theorem in Section~\ref{sec:Property_A}.

In Sections \ref{sec:Artinfit}, \ref{sec:c3t6} and \ref{sec:c6} we show that almost-large-type Artin groups, finitely presented graphical $C(3)$--$T(6)$ groups, and finitely presented classical $C(6)$ groups act geometrically on uniformly locally finite fit systolic complexes, thus proving Corollaries A, B and C.
In Section \ref{sec:54547272} we exhibit an example of a fit systolic $3$-pseudomanifold acted geometrically upon a group, proving Corollary D. 

Section~\ref{sec:preliminaries} contains some preliminary material on graphs (fixing the notation for the rest of the paper), and on Yu's Property A. In the final Section~\ref{sec:final} we comment on further directions of research and applications of our results.

\section{Preliminaries}\label{sec:preliminaries}

\subsection{Graphs}
\label{sec:premgr}
All graphs $G=(V(G),E(G))$ considered in this paper are simplicial, undirected and connected.
For two distinct vertices $u,v\in V(G)$ we write $u\sim v$ when there is an edge connecting $u$ with $v$ and $u\nsim v$ otherwise. 
The subgraph of $G$ \emph{induced by} a subset $A\subseteq V(G)$ is the graph
$G[A]=(A,E')$ such that $uv\in E'$ if and only if $uv\in E(G)$. A \emph{clique} is a complete subgraph of $G$, i.e., a subset of pairwise adjacent vertices. 
The \emph{distance}
$d(u,v)=d_G(u,v)$ between two vertices $u$ and $v$ of a graph $G$ is the
length of a shortest $(u,v)$-path. 
An induced subgraph $H=G[A]$ of a graph $G$ is an \emph{isometric subgraph} of $G$ if $d_H(u,v)=d_G(u,v)$ for any
two vertices $u,v\in A$. The \emph{interval}
$I(u,v)$ between $u$ and $v$ consists of all vertices on shortest
$(u,v)$-paths, that is, of all vertices (metrically) \emph{between} $u$
and $v$: $I(u,v)=\{ x\in V(G): d(u,x)+d(x,v)=d(u,v)\}$. An induced
subgraph of $G$ 
is called \emph{convex}
if it contains the interval in $G$ between any two of its
vertices.  For a vertex $v$ of $G$ and an integer $r\ge 1$, we will denote  by $B_r(v)$ the \emph{ball} in $G$
(and the subgraph induced by this ball)  of radius $r$ centred at  $v$, i.e., $B_r(v)=\{ x\in V(G): d(v,x)\le r\}.$ More generally, the $r$-{\it ball  around a set} $A\subseteq V(G)$
is the set (or the subgraph induced by) $B_r(A)=\{ v\in V(G): d(v,A)\le r\},$ where  $d(v,A)=\mbox{min} \{ d(v,x): x\in A\}$.
Similarly, the \emph{sphere} in $G$ centred at $v$ of radius $r$ is the set (or the subgraph induced by) $S_r(v)=\{ x \in V(G): d(v,x) = r\}$.
As usual, $N(v)=S_1(v)$ denotes the set of neighbours of a vertex
$v$ in $G$, and $N[v]=B_1(v)$.
The \emph{link} of $v \in V(G)$ is the subgraph of $G$ induced by $N(v)$.  Any subset $S$ of the sphere $S_r(v)$ is said to be at \emph{uniform distance $r$} from $v$. 

Three vertices $u,v,w$ of a graph $G$ form a \emph{metric triangle} $uvw$ if   $I(u,v), I(u,w),$ and $I(w,u)$ pairwise intersect only at the end-vertices. A metric triangle $uvw$ is \emph{strongly equilateral} of size $k$ if $d(u,x)=k$ for all $x\in I(v,w)$ (and similarly with the roles of $u,v,w$ swapped). 
A \emph{quasi-median} of a triplet $x,y,z$ is a metric triangle $x'y'z'$ such that $d(x,y)=d(x,x')+d(x',y')+d(y',y), d(y,z)=d(y,y')+d(y',z')+d(z',z),$ and $d(z,x)=d(z,z')+d(z',x')+d(x',x)$. Each triplet $x,y,z$ admits at least one quasi-median. 

\begin{definition}[Bridged/systolic]
    A graph $G$ is \emph{bridged} if any isometric cycle of $G$ has size $3$. A flag simplicial complex is \emph{systolic} if its $1$-skeleton is a bridged graph.
\end{definition}

Systolic complexes are exactly simply connected flag simplicial complexes in which the links of vertices do not contain induced $4$-cycles and $5$-cycles \cite{Ch_CAT, JaSw}.
A \emph{chord} of a cycle in a graph $G$ is an edge of $G$ that is not part of the cycle but connects two of its vertices.
A graph $G$ is \emph{chordal} if any cycle of length $>3$ of $G$ has a chord. In other words, every such cycle is not an induced subgraph.  Obviously, chordal graphs are bridged. 

A \emph{$k$-fan} is a graph consisting of a path of length $k$ and a vertex adjacent to all vertices of this path; Figure~\ref{f:fans}, top left. A \emph{tripod} is a star with 3 tips, i.e., a graph consisting of a central vertex adjacent to $3$ pairwise non-adjacent vertices (called \emph{tips}); Figure~\ref{f:fans}, top centre. An \emph{asteroid} is a graph consisting of a triangle and each vertex of this triangle is adjacent to one extra-vertex so that the 3 extra-vertices are pairwise non-adjacent; Figure~\ref{f:fans}, top right. 
An \emph{$n$-tripod} is the graph obtained by subdividing each of the edges of a tripod into paths of length $n$; Figure~\ref{f:fans}, bottom left.
Similarly, an \emph{$n$-asteroid} is the graph obtained from an asteroid by subdividing each of the external three edges into paths of length $n$; Figure~\ref{f:fans}, bottom right. If we do not specify $n$, we refer to $n$-asteroids and $n$-tripods as \emph{large asteroids} and \emph{large tripods}. 
We further abuse notation and assume that a $0$-tripod is a vertex and a $0$-asteroid is a triangle.
A \emph{doubled asteroid} is the graph obtained from glueing two asteroids along their tips; see Figure \ref{f:triplane}, right.

\begin{figure}[h!]
	\begin{center}
	\includegraphics[scale=1]{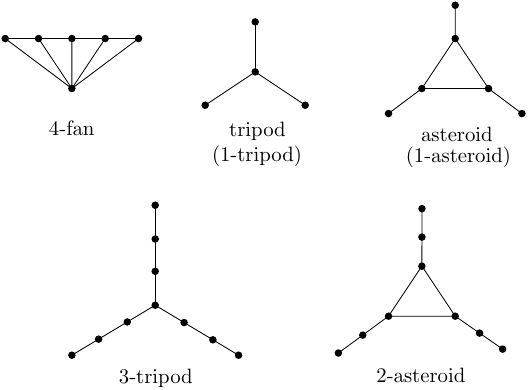}
	\end{center}
	\caption{}
	\label{f:fans}
\end{figure}

\subsection{Yu's Property A}\label{sec:A}
For completeness and for a curious reader we present here a definition of Property A that is basically the same as the original Yu's definition from \cite{Yu2000}. It can be seen as a ``nonequivariant'' version of the definition of amenability via F\o lner sets. See e.g.\ \cite{Willett2009} for dozens of other definitions, and \cite{NWZ2025} for a recent definition that resembles even more closely the one of amenability.

\begin{definition}\label{d:A}
A uniformly discrete metric space $(X,\rho)$ has the
\emph{Property A} if for every $\epsilon,R > 0$ there exists a collection of finite subsets
$\{ A_x \}_{x\in X}$, $A_x\subset X\times \mathbb{N}$, and a constant $S > 0$ such that
\begin{itemize}
    \item $\frac{|A_x\triangle A_y|}{|A_x\cap A_y|}\leq \epsilon$ if $\rho(x,y)\leq R$, and
    \item $A_x\subseteq B_S(x)\times \mathbb{N}$, where $B_S(x)$ is the $S$-ball in  $(X,\rho)$ around $x$.
\end{itemize}
\end{definition}

We continue the list of characterizations of Property A for a finitely generated group $H$ from  Introduction:
\begin{itemize}
        \item the reduced $C^\ast$-algebra of $H$ is exact \cite{KiWas};
		\item the uniform Roe algebra of $H$ is nuclear \cite{Oz2000};
        \item bounded cohomology groups of $H$
with coefficients in a certain class of dual modules vanish \cite{BNN2012,Mon2011};
        \item $H$ has the operator norm localization property \cite{Sa2014};
        \item asymptotically invariant cohomology groups of $H$ vanish \cite{BNW2012}.
	\end{itemize}
The list is very incomplete.

For the proof of our Main Theorem all we will need to know about Property A is a useful combinatorial characterization by \v{S}pakula and Wright \cite{SpWr}---see Proposition~\ref{prop:criterion} in Section~\ref{sec:Property_A}.

\section{Fit systolic complexes}
	\label{sec:fit}
	In this section we define a subclass of systolic complexes---fit systolic complexes---which is the main object of study in this paper. To define it we need some preparations.

\subsection{Level sets and fit systolic complexes}\label{sec:levelsfit}

First, we recall some known properties  of bridged graphs.

\begin{theorem}\label{bridged-convex-balls}  \cite{FaJa,SoCh} A graph is bridged if and only if   the balls $B_r(A)$ around convex sets $A$ are convex. 
\end{theorem} 


\begin{lemma} \cite{Ch_mt} \label{strongly-equilateral} All metric triangles of a bridged graph are strongly equilateral. 
\end{lemma}

Recall, that for $B,C>0$ a map $f\colon (X,d_X) \to (Y,d_Y)$ between metric spaces is a \emph{$(B,C)$-quasi-isometry} if the $C$-neighbourhood of the image of $f$ is the whole space $Y$, and for every $x,x'\in X$ we have
\begin{align*}
    B^{-1}d_X(x,x')-C\leq d_Y(f(x),f(x'))\leq Bd_X(x,x')+C.
\end{align*}

\begin{lemma} \label{chordal-quasi-tree} \cite{BrChDr}  Chordal graphs are $(1,2)$-quasi-isometric to trees. 
\end{lemma}


A graph $G$ satisfies the \emph{triangle condition} (TC) if for  $v,x,y\in V$ with $d(v,x)=d(v,y)=k$ and $x\sim y$, there exists $z\sim x,y$ with $d(v,z)=k-1$. Bridged graphs satisfy (TC), in fact they satisfy the following stronger property:

\begin{lemma} \label{l:clique-in-the-sphere} \cite{JaSw} Let $G=(V,E)$ be a bridged graph, $x$ be any vertex of $G$, and $K$ be a clique of $G$ at uniform distance $k$ from $x$. Then the set of all vertices $y$ at distance $k-1$ from $x$ and adjacent to all vertices of $K$ is a nonempty clique of $G$.   
\end{lemma}

For two vertices $u, v\in V$  with $d(u,v)=\ell$ and $0\le k\le \ell$ 
the \emph{$k$-level} $L_k(u,v)$ of the interval $I(u,v)$ is the subgraph of $G$ induced by all  $x\in I(u,v)$ such that $d(u,x)=k$ and $d(x,v)=\ell-k$. Convexity of balls (Theorem \ref{bridged-convex-balls}) implies that bridged graphs also satisfy the following property (called the \emph{interval neighbourhood condition} and abbreviated (INC) \cite{ChChGi}): the neighbours of $u$ in  $I(u,v)$ induce a clique of $G$; in other words the level $L_1(u,v)$ in $I(u,v)$ is a clique. 

The following result was proved in \cite[Lemma 8.4, Corollary 8.6]{OsPr} (and was recently independently reproved by Steven Seif).
 It implies---via Lemma~\ref{chordal-quasi-tree}---that level sets in bridged graphs are $(1,2)$-quasi-trees (see Figure~\ref{f:fit} left for a schematic picture of such level set). We present a simple alternative proof.

\begin{proposition} \label{levels-are-chordal} For a bridged graph $G$ and $u,v\in V$, each level $L_k(u,v)$ of $I(u,v)$ is a convex chordal subgraph of $G$ not containing induced 3-fans. 
\end{proposition} 

\begin{proof} $L_k(u,v)$ is convex as the intersection of the convex balls 
$B_k(u)$ and $B_{\ell-k}(v)$. Hence $L_k(u,v)$ is bridged. Now, the proposition follows from the following two claims.

\begin{claim} 
$L_k(u,v)$ does not contain $3$-fans.
\end{claim}
\begin{proof}
Suppose not, and consider a $3$-fan $F\subseteq L_k(u,v)$ defined by the 
3-path $P=(p,q,q',p')$ and a vertex $x$ adjacent to 
all vertices of $P$. Applying 
Lemma \ref{l:clique-in-the-sphere} to the  
triangles $x,p,q$ and $x,p',q'$ of $F$, 
we deduce 
that there exist vertices $y\sim p, q, x$ and 
$y'\sim x,p',q'$ 
at distance $k-1$ from $u$. Since $F$ is induced 
and the balls 
of $G$ are convex, the vertices $y$ and $y'$ are different 
and $y\nsim p',q'$ and $y'\nsim p, q$. Since $y,y'$ 
are both adjacent
to $x$, the convexity of the ball $B_{k-1}(u)$ implies 
that $y\sim y'$. But then the vertices $y,y',q',q$ 
induce a forbidden $4$-cycle. This shows that $L_k(u,v)$ does 
not contain $3$-fans.
\end{proof}
%

\begin{claim} 
If a bridged graph $H$ contains an induced cycle $C$ of length $>3$, then $H$ contains an induced $3$-fan.
\end{claim} 

\begin{proof}
Since $C$ has length $>3$, $C$ is not isometric. Let $a,b$ be a closest pair of vertices of $C$ such
that $d_H(a,b)<d_C(a,b)$. Let $x$ be a neighbour of $a$ on a shortest path $P$ among the two $(a,b)$-paths of $C$. By the choice of $a,b$, $d_H(b,x)=d_C(b,x)$. Let $y$ be the second neighbour of $x$ in $P$. Then $y\in I(x,b).$ 
If $a\in I(x,b)$, by (INC), $y\sim a$ and we get a contradiction with the assumption that $C$ is induced. Thus
$d_H(b,x)=d_H(b,a)$. By (TC), there exists $z\sim x,a$ one step closer to $b$ than $a,x$. By (INC), $z\sim y$
and by (TC) there exists $w\sim y, z$ one step closer to $b$ than $y,z$. We assert that $a,x,z,y,w$ induce a 3-fan.
From the choice of $a,x$ and the definition of $z$ and $w$, we deduce that $w$ cannot be adjacent $a$
or $x$. Since $C$ is induced, $a$ and $y$ are not adjacent. This establishes the claim. 
\end{proof} 
\end{proof}

\begin{definition}[Fit]\label{def:fit}
    A systolic complex $X$ is \emph{fit} if there exist $B,C>0$ such that every level (in every interval) is $(B,C)$-quasi-isometric to a segment.
\end{definition}

\begin{figure}[h!]
    \begin{center}
	\includegraphics[scale=0.2]{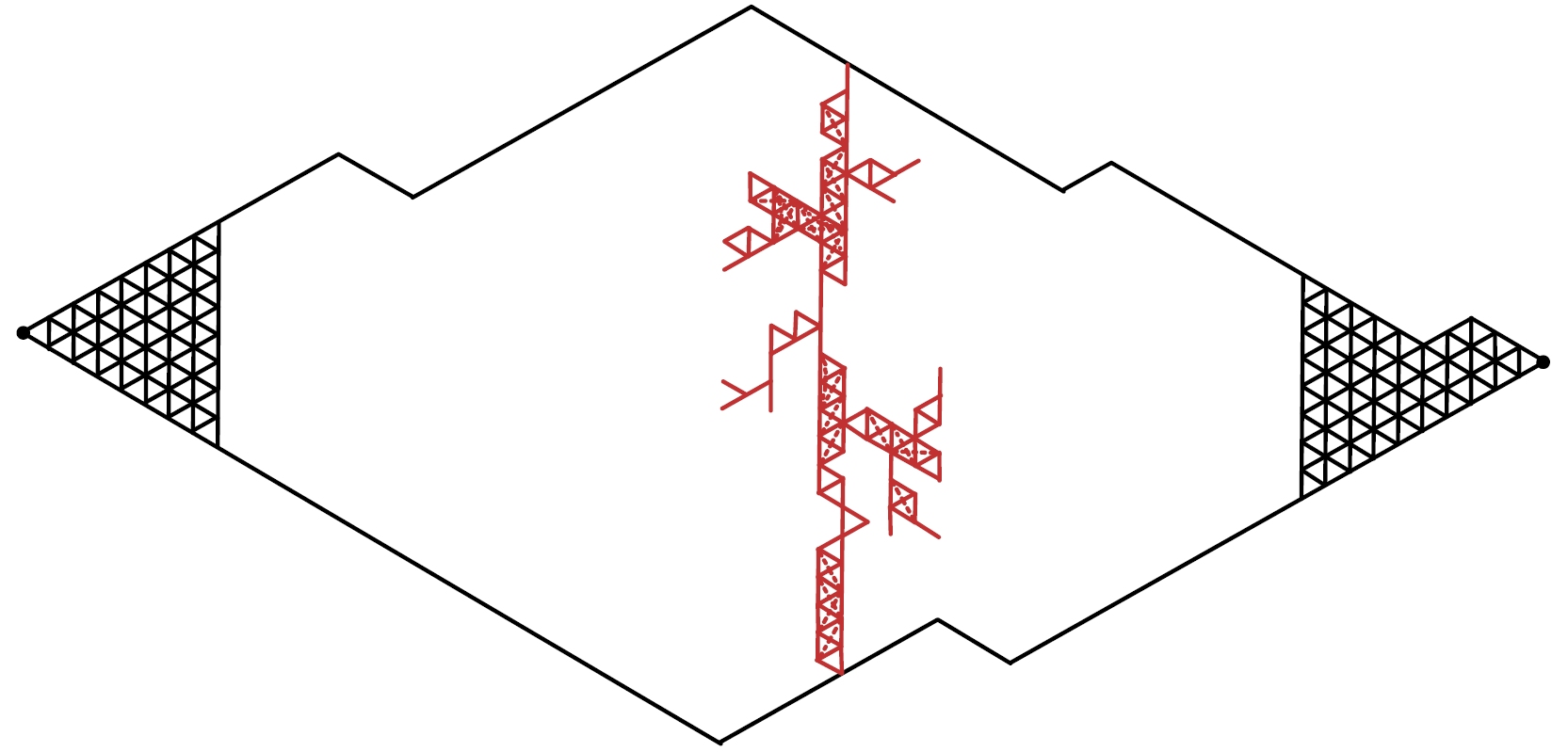}
    \end{center}
    \caption{A level set (red) in an interval in a systolic complex/bridged graph. Being ``fit'' means that level sets are ``not too large''.}
    \label{f:fit}
\end{figure}	

\begin{remark}
    \label{rem:2dim}
    Observe that $2$-dimensional systolic complexes are fit.
    In fact, their intervals embed isometrically into a systolic flat plane \cite[Lemma 3]{ChLaRa}.
    Note that the fit systolic complexes associated (below) to large-type Artin groups (Section~\ref{sec:Artinfit}) as well as $C(3)$--$T(6)$ (Section~\ref{sec:c3t6}), and $C(6)$ groups (Section~\ref{sec:c6}) are in general not $2$-dimensional.

\end{remark}

\begin{remark}
    \label{rem:generalfit}
    It is worth mentioning that for our main theorem to hold, one can weaken the definition of fit systolic complex.
    It is enough to have a point $x_0\in X$ and a polynomial $p$ such that for every $x\in X$ any ball of radius $r$ in any level set of $I(x,x_0)$ has at most $p(r)$ vertices. In the case of the main examples in this paper---in Sections~\ref{sec:Artinfit},~\ref{sec:c3t6},~\ref{sec:c6}, and ~\ref{sec:54547272}---as well as for $2$-dimensional systolic complexes (see Remark~\ref{rem:2dim} above) the stronger property from Definition~\ref{def:fit} holds.
\end{remark}

\subsection{Characterizations of fit systolicity}\label{sec:charact_fit}

We present two more structural characterizations of fit systolic complexes (Proposition~\ref{prop:equivalent_definition}) and a following condition for fitness (Corollary~\ref{c:fit_condition}) that are easier to check. We use them in Sections~\ref{sec:Artinfit}, \ref{sec:c6}, and \ref{sec:54547272}. First we need some definitions.
Take a sequence of tetrahedra
$\sigma_i = \sigma(v_i, a_i, b_i, c_i)$ and $\sigma_i' = \sigma(v_{i-1}, a_i, b_i, c_i)$, for $i \in \mathbb{Z}$ and glue ``zipped'' systolic half-planes 
along the geodesics $\alpha$, $\beta$, $\gamma$, spanned by the sets of vertices $\{v_i, a_i: i \in \mathbb{Z}\}$, $\{v_i, b_i: i \in \mathbb{Z}\}$, and $\{v_i, c_i: i \in \mathbb{Z}\}$, as in Figure \ref{f:triplane}, and call the resulting systolic complex a \emph{triplane}. 
An \emph{$n$-arrowhead} is the graph obtained by looking at the interval $I(v_0,v_n)$ between $v_0$ and $v_n$ in a triplane (see Figure \ref{f:triplane}). The sequence of tetrahedra is called the \emph{axis} of the arrowhead (respectively, of the triplane). 

\begin{figure}[h!]
    \begin{center}
	\includegraphics[trim={0cm 0 0cm 3.5cm},clip,scale=0.6]{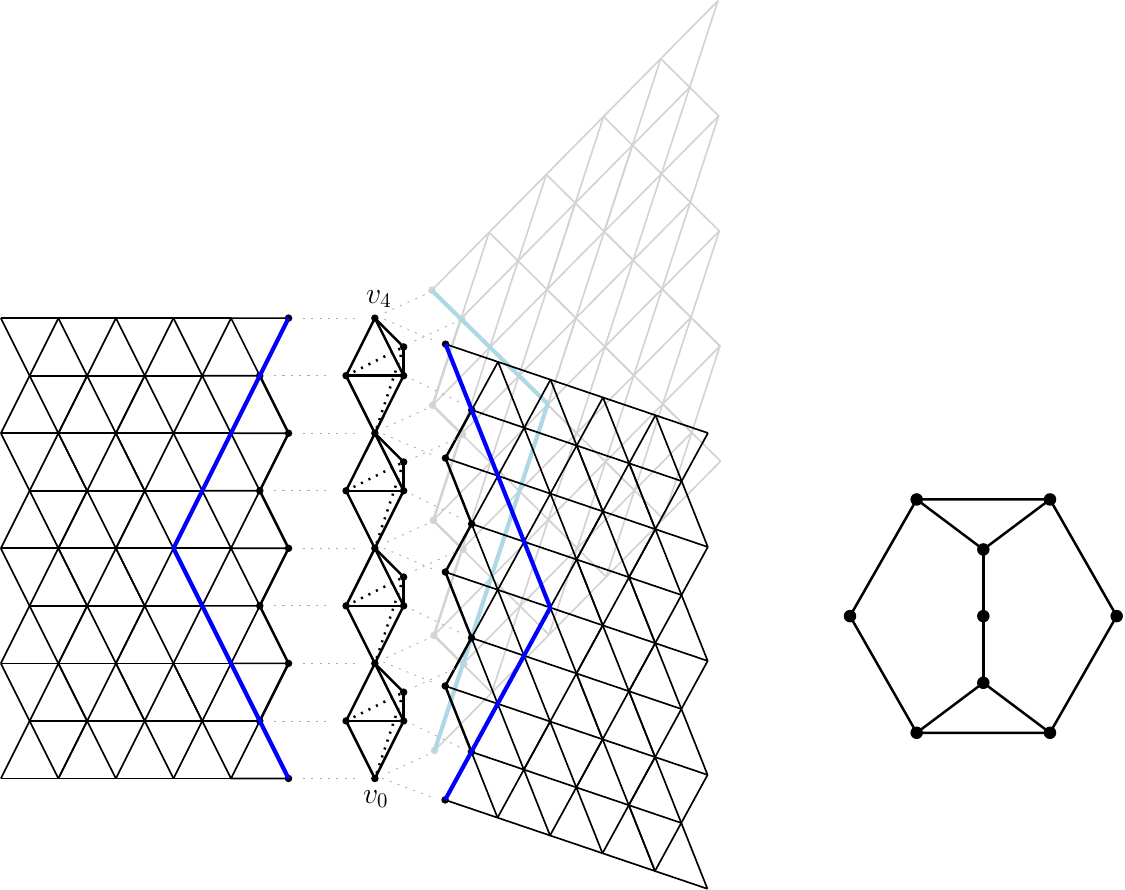}
    \end{center}
    \caption{On the left: a triplane with half-planes glued in a ``zipped'' way along a sequence of tetrahedra. A $4$-arrowhead is outlined in thick blue. On the right: the link of a vertex $v_i$, a doubled asteroid.}
    \label{f:triplane}
\end{figure}	

Given $N\in\mathbb{N}$ we say that a graph is \emph{$N$-branchless} if it does not contain isometrically embedded $n$-tripods or $n$-asteroids with $n>N$; and \emph{$M$-harmless} if it does not contain isometrically embedded $m$-arrowheads with $m>M$.

\begin{lemma}\label{lem:tripod-asteroid}
    Let $u,v$ be two vertices of a systolic complex $X$, with $d(u,v)=\ell$. Let $1\leq k \leq \ell-1$.
    \begin{enumerate}
        \item \label{lem:tripod-asteroid_1} If the level set $L_k(u,v)$ contains an isometrically embedded $n$-tripod then $L_{k+1}(u,v)$ contains an isometrically embedded $(n-1)$-asteroid.
        \item \label{lem:tripod-asteroid_2} If the level set $L_k(u,v)$ contains an isometrically embedded $n$-asteroid then $L_{k+1}(u,v)$ contains an isometrically embedded $n$-tripod. 
    \end{enumerate}
\end{lemma}
\begin{proof}
    We begin by proving item (\ref{lem:tripod-asteroid_1}). 
    Let $T$ be an $n$-tripod in $L_k(u,v)$ with centre $v_0$ and neighbours $x,y,z$ of $v_0$ in $T$.
    By the triangle condition (TC), applied to $v$ and each of the edges $xv_0$,$yv_0$, and $zv_0$, we find three vertices $a_1,b_1,c_1\in L_{k+1}(u,v)$ such that $a_1\sim x,v_0$, $b_1\sim y,v_0$, and $c_1\sim z,v_0$.
    Since $x,y,z$ are not adjacent and $L_k(u,v)$ is convex, the vertices $a_1,b_1,c_1$ are pairwise different.
    Since $a_1,b_1,c_1\sim v_0$ and $L_{k+1}(u,v)$ is convex, $a_1,b_1,c_1$ are pairwise adjacent, so we obtain a triangle.
    Now we apply (TC) to $v$ and the edges of remaining three paths of length $n-1$ of $T$ (with ends at $x,y,$ and $z$) to derive new vertices in $L_{k+1}(u,v)$.
    By convexity of $L_{k+1}(u,v)$ these vertices are pairwise distinct and induce three paths, which together with the triangle $a_1b_1c_1$ define an $(n-1)$-asteroid $A$ in $L_{k+1}(u,v)$.
    Since the level $L_{k+1}(u,v)$ is convex and the $n$-tripod $T$ is isometric, the  $(n-1)$-asteroid $A$ is also isometrically embedded in $X$.

    The proof of item \ref{lem:tripod-asteroid_2} is similar, the only difference being that at the first step, instead of (TC), we apply Lemma \ref{l:clique-in-the-sphere} to $v$ and the triangle of the $n$-asteroid.
    Then just as before we obtain an isometrically embedded $n$-tripod in $L_{k+1}(u,v)$.
\end{proof}

\begin{proposition}\label{prop:equivalent_definition}
    Let $X$ be a systolic complex. Then the following are equivalent:
    \begin{enumerate}
        \item $X$ is fit;
        \item There exists $N\in \mathbb{N}$ such that all level sets of all intervals are $N$-branchless.
        \item There exists $M\in \mathbb{N}$ such that $X$ is $M$-harmless.
    \end{enumerate}
    Furthermore, if all intervals are $N$-branchless then all level sets are  $(1,2N+2)$-quasi-isometric to segments.
\end{proposition}

\begin{proof}
    We first show that (2) and (3) are equivalent.
    It is clear that (2) implies (3), as  level sets in the interval $I(v_0,v_{2n+1})$ between opposite endpoints of an $2n+1$-arrowhead  contain $n$-tripods and $n$-asteroids.
    Now suppose $X$ is $M$-harmless, but there are vertices $u,v$ such that some  level $L_k(u,v)$ contains an isometrically embedded $\lceil\frac{M+1}{2}\rceil$-tripod or $\lceil\frac{M}{2}\rceil$-asteroid.
    By iteratively applying Lemma \ref{lem:tripod-asteroid} we obtain an 
    $(M+1)$-arrowhead $Y$ whose axis is the defined by a sequence of tetrahedra starting with $\sigma'_1 = \sigma(v_0=v, a_1, b_1, c_1)$ and $\sigma_1 = \sigma(v_1, a_1, b_1, c_1)$.
    By construction, $Y$ is systolic. 
    
    For a level $L_i(u,v)$, we denote by $Y_i$ the intersection $Y\cap L_i(u,v)$. If this intersection is non-empty, then  $Y_i$ is  a large asteroid or a large tripod.
    From the construction of $Y$ it follows that if $y$ is any vertex of $Y$, say $y\in Y_j$ with $k<j$, then there is an $i$ with $k\le i<j$ such that $Y_i$ contains two vertices $a$ and $b$, and the vertices $y,a,b$ define an equilateral triangle $\Delta(y,a,b)$ of side $\ell=j-i$ in $Y$ (i.e., a subdivision of the equilateral triangle with side $\ell$ into equilateral triangles with side 1).
    Then we say that the vertices $a,b\in L_i(u,v)$ \emph{generate} the vertex $y$.
    Vice-versa, any pair of vertices $a,b\in Y_i$, where $i\ge k$, generate a unique vertex $y$, which belongs to the level $L_{i+d(a,b)}(u,v)$. 
    
    To obtain a contradiction, it remains to show that the  $(M+1)$-arrowhead $Y$ is isometrically embedded in $X$. Pick any two vertices $x,y$ of $Y$, say $x\in L_i(u,v)$ and $y\in L_j(u,v)$ with $k\le i\le j\le d(u,v)$. We proceed by double induction on $\ell=j-i$ and on the distance $d_Y(x,y)$ between $x$ and $y$ in $Y$.  If $\ell=0$, then $x$ and $y$ belong to the same level $L_i(u,v)$ and we are done because $Y_i$  is an isometrically embedded large tripod or large asteroid.  Now suppose that $\ell>0$. Let $z$ be a closest to $y$ vertex of $Y_i$. If there are several such vertices in $Y_i$, suppose additionally that $z$ is chosen to be closest to $x$. By the definition of $Y$, it follows that the distance between $y$ and $z$ in $X$ is $\ell$ and $Y$ contains a  $(z,y)$-path of this length, traversing the level sets  $L_i(u,v),L_{i+1}(u,v), \ldots, L_j(u,v)$. Consequently, $d_Y(z,y)=d(z,y)=j-i=\ell$.  If $x=z$, then we are done. Therefore, let $x\ne z$. 

    From the definition of $z$ it follows that  $y$ is generated by the pair $z,w$, where $w$ is some vertex of $Y_i$ such that $z$ is on a $(w,x)$-path in $Y_i$. Furthermore, if $x\nsim z$, then the pair $w,x$ generates a vertex $y_0\in Y_{j'}$, where $j'=j+d(z,x)$. Note that in this case the equilateral triangle $\Delta(y,z,w)$ is contained in the equilateral triangle $\Delta(y_0,x,w)$. Notice also that in $X$ we have $y\in I(y_0,w)\cap I(y_0,z)$ and $y_0\in I(y,v)$. Furthermore,  $z$ belongs to a shortest $(x,y)$-path $\pi$ in $\Delta(y_0,x,w)$, first traversing the level sets and then going straight along $Y_i$. We denote by $\gamma=(x,q,q',\ldots,w)$ the unique shortest path of $Y_i$ connecting $x$ and $w$ (and passing via $z$). We distinguish two cases.

 \begin{case}
        The vertices $x$ and $z$ are not adjacent.
    \end{case}

    For all $p\in \pi\cap \gamma$ different from $x$, we have that $d_Y(y,p)<d_Y(y,x)$, so by the second inductive hypothesis we conclude that $d_Y(p,y)=d(p,y)$. Consider the first two vertices $q$ and $q'$ of 
    $\pi\cap \gamma$.  Since 
    $x\nsim z$, the vertex $q'$ exists. 
    Because $x\sim q$, we have $d(q,y)-1\leq d(x,y)\leq d(q,y)+1$. Therefore, to prove $d(x,y)=d_Y(x,y)$, and since $d_Y(x,y)=d_Y(q,y)+1$, it suffices to rule out the possibilities $d(x,y)=d(q,y)-1$ and $d(x,y)=d(q,y)$. The first case is immediate, because if $d(x,y)=d(q,y)-1$, then $x,q'\in I(q,y)$. Since $X$ is systolic, $x$ and $q'$ must be adjacent, a contradiction to the assumption that $\pi$ is a shortest $(x,y)$-path of $Y$. 
    Thus we may assume that $d(x,y)=d(q,y) \coloneq k$.
    Let $x_1$ be the common neighbour of $x$ and $q$ in $\Delta(y_0,w,x)$. Since $x_1\in L_{i+1}(u,v)$, by the first induction assumption $d(x_1,y)=d_Y(x_1,y)=k$. 
    By Lemma \ref{l:clique-in-the-sphere} in $X$ applied to $y$ and the triangle $x,q,x_1$, there is a vertex $s\sim x,q,x_1$ at distance $k-1$ from $y$. 
    Since $s,q'\in I(q,y)$, the vertices $s$ and $q'$ must be adjacent. 
    Since $s\sim q',x\in L_i(u,v)$ and $L_i(u,v)$ is convex, this implies that $s\in L_i(u,v)$. Since $Y_i$ is a large tripod or large asteroid, $s$ is not a vertex of $Y$. 
    Now let $q_1$ be the common neighbour of $q'$ and $q$ in $\Delta(y_0,w,x)$. Obviously, $q_1\ne s$. 
    Since $q_1,s\in I(q,y)$, the vertices $s$ and $q_1$ must be adjacent. 
    By Lemma \ref{l:clique-in-the-sphere} in $X$  applied to $y$ and the triangle $q',s,q_1$, there is a vertex $s'\sim q',s,q_1$ at distance $k-2$ from $y$. By the construction of $Y$, $s'$ does not belong to $\Delta(y_0,w,x)$. 
    Let $x_2$ be the common neighbour of $q_1$ and $x_1$ in $\Delta(y_0,w,x)$, and let $q_2$ be the other neighbour of $q_1$ in $\Delta(y_0,w,x)\cap L_{i+2}(u,v)$. Clearly, $s'\ne q_2$. 
    Since $s',q_2\in I(q_1,y)$, the vertices $s'$ and $q_2$ must be adjacent. 
    Consider the 5-cycle $ss'q_2x_2x_1$ in $X$. 
    Since $s\in L_i(u,v)$ and $q_2,x_2\in L_{i+2}(u,v)$, the vertex $s$ cannot be adjacent to $q_2$ or $x_2$. 
    Since $d(y,s')=d(y,q_2)=k-2$ and $d(y,x_1)=k$, the vertex $x_1$ cannot be adjacent to $s'$ or $q_2$. 
    Since $x_1, s'\in L_{i+1}(u,v)$ and $L_{i+1}(u,v)$ is convex, $x_2$ cannot be adjacent to $s'$.
    Therefore $ss'q_2x_2x_1$ is an induced 5-cycle in  $X$, which is impossible since 
    $X$ is systolic.  This concludes the analysis of the case $x\nsim z$.

\begin{case} The vertices $x$ and $z$ are adjacent. 
\end{case}

    In this case we have $z=q$. 
    Since $d_Y(z,y)=\ell$ and $x\sim z$, we have $d_Y(x,y)=\ell+1$ and $\ell \leq d(x,y)\leq \ell+1$. To prove that $d(x,y)=d_Y(x,y)$, we argue by contradiction and assume that $d(x,y)=\ell$. We will derive a contradiction by constructing an $(x,y)$-path in $Y$ of length $\ell$, contradicting $d_Y(x,y)=\ell+1$. 
    Let $\alpha=(z=y_0,y_1,\ldots,y_{\ell-1},y_{\ell}=y), \gamma=(z=q,q',q'',\ldots,w),$  and  $\beta=(w=t_0,t_1,\ldots,t_{\ell-1},t_{\ell}=y)$ be the geodesic paths of $\Delta(y,z,w)$ connecting the pairs $\{ z,y\}$, $\{ z,w\}$, and $\{ w,y\}$.
    By (TC) in $X$ applied to $y$ and the edge $zx$, we find a vertex $p_1\sim z,x$ at distance $\ell-1$ from $y$. If $p_1$ belongs to $Y$, applying the induction assumption to the pair $p_1,y$ we will find a $(p_1,y)$-path of length $\ell-1$ in $Y$.
    Adding the edge $xp_1$ to this path, we generate a $(x,y)$-path of length $\ell$, as required.
    Thus suppose that $p_1$ does not belong to $Y$.
    Hence $p_1\ne y_1$. Since $p_1,y_1\in L_{i+1}(u,v)$ and $z\sim p_1,y_1$, we conclude that $p_1\sim y_1$. 
    By (TC) in $X$ applied to $y$ and the edge $p_1y_1$, we find a vertex $p_2\sim p_1,y_1$ at distance $\ell-2$ from $y$.
    If $p_2\in Y$ and $p_2\ne y$, we can apply the induction hypothesis to the pairs of vertices $y,p_2$ and $p_2,x$ and deduce that they are connected in $Y$ by paths of lengths $\ell-2$ and $2$.
    Combining those paths, we get a $(y,x)$-path of length $\ell$ in $Y$ and we are done.
    So, suppose $p_2\notin Y$ and thus $p_2\ne y_2$.
    Again, from the convexity of $L_{\ell+2}(u,v)$ and $y_1\sim y_2,p_2$ we infer that $p_2\sim y_2$.
    Continuing, this way we construct a $(x,y)$-path $\lambda=(x=p_0,p_1,p_2,\ldots, p_{\ell-1},p_\ell=y)$ of length $\ell$ such that each vertex 
    $p_k$ for $k\ne 0,\ell$ does not belong to $Y$ and is adjacent to the vertices $y_{k-1},y_k$ of the path $\alpha$. 
    
    The vertex $p_{\ell-1}$ is adjacent to $y=p_\ell$.
    Since $y$ is adjacent to $t_{\ell-1}\in L_{j-1}(u,v)$, the convexity of $L_{j-1}(u,v)$ implies that $p_{\ell-1}\sim t_{\ell-1}$. Let $s_{\ell-2}$ be the common neighbour of $y_{\ell-1}$ and $t_{\ell-1}$ in $\Delta(y,z,w)\cap L_{j-2}$.
    Notice that $s_{\ell-2}$ is also adjacent to $y_{\ell-2}$ and $t_{\ell-2}$. 
    The vertices $p_{\ell-1},y_{\ell-2},s_{\ell-2},t_{\ell-1}$ define a 4-cycle of $X$, which cannot be induced because $X$ is systolic.
    By inductive hypothesis, $t_{\ell-1}$ cannot be adjacent to $y_{\ell-2}$.
    This implies that $p_{\ell-1}\sim s_{\ell-2}$.
    Since $p_{\ell-1}\sim p_{\ell-2}$, the convexity of $L_{\ell-2}(u,v)$ implies that $p_{\ell-2}\sim s_{\ell-2}$. If 
    $s_{\ell-3}$ is the common neighbour of $s_{\ell-2}$ and $y_{\ell-2}$ in $\Delta(y,z,w)$, then employing the same argument, we deduce that $s_{\ell-3}\sim p_{\ell-2},x_{\ell-3}$.
    Continuing this way, we deduce that the vertex $s_1\in L_{i+1}(u,v)$, which  in $\Delta(y,z,w)$ is the common neighbour
    of $s_2,y_2$ as well as of $q',q''\in \gamma$, must be adjacent to the vertex $p_1$. Again, in $X$ we get the 4-cycle $p_1zq's_1$, which cannot be induced.
    Consequently, $p_1$ must be adjacent to $q'$. Since the level $L_i(u,v)$ is convex and $p_1\sim q',x$, we deduce that $x\sim q'$. Therefore $Y_i$ is a large asteroid having the triangle $xzq'$ as the centre. But then, by the construction of $Y$, the unique common neighbour $y_1$ of $z,q'$ in $L_{i+1}(u,v)$ also must be adjacent to $x$.
    Therefore, we get a $(x,y)$-path in $Y$ of length $\ell$, as required. 
    
    Since in both cases we obtain $d_Y(x,y)=d(x,y)$, the inductive step implies that this equality holds for all pairs of vertices in $Y$. Hence $Y$ is an isometrically embedded $(M+1)$-arrowhead, which is impossible.  Consequently, all level sets of all intervals in $X$ are $(\lceil\frac{M}{2}\rceil)$-branchless. This completes the proof that  (2) and (3) are equivalent.

    \medskip    

    Now we show that (1) and (2) are equivalent.
    It is clear that if level sets of intervals contain arbitrarily large tripods or asteroids, then $X$ is not fit. 
    We will show that if there exists $N\in \mathbb{N}$ such that all level sets of all intervals are $N$-branchless, then $X$ is fit. Moreover, we show that the level sets of intervals in $X$ are then $(1,2N+2)$-quasi-isometric to segments.

    It suffices to prove that any finite chordal graph $R$ that is $N$-branchless and does not contain $3$-fans is quasi-isometric to a segment, and that the quasi-isometry constants depend appropriately on $N$.
    Pick any pair $x,y$ of vertices of $R$ and let $P$ be any shortest $(x,y)$-path. First we prove that the interval $I(x,y)$ is quasi-isometric to $P$.
    
    \begin{claim} \label{c:distance-interval} Any vertex $z$ of $I(x,y)$ is at distance at most 1 from $P$. 
    \end{claim}

    \begin{proof} Let $z'$ be a vertex of $P$ such that $d(x,z)=d(x,z')$ and $d(y,z)=d(y,z')$.
    Consider any quasi-median of the triplet $x,z,z'$.
    Since $z,z'$ belong to the same level of $I(x,y)$, this quasi-median has the form $x'zz'$, where $x'\in I(x,z)\cap I(x,z')$.
    Analogously, any quasi-median of the triplet $y,z,z'$ has the form $y'zz'$ for $y'\in I(y,z)\cap I(y,z')$.
    Pick any shortest paths $P(x',z),P(x',z'),P(y',z),P(y',z')$ between the pairs $\{x',z\}$, $\{ x',z'\}$, $\{ y',z\}$, and $\{ y',z''\}$, respectively.
    From the choice of the vertices  $x',y',z,z'$ it follows that these four paths pairwise intersect only at their common end-vertices.
    Consequently, their union is a simple cycle $C$ of $I(x,y)$, which cannot be induced.
    Since there are no edges between the vertices of $(P(x',z)\cup P(x',z'))\setminus \{z,z'\}$ and the vertices of $(P(y',z)\cup P(y',z'))\setminus \{ z,z'\}$, either the vertices $z$ and $z'$ are adjacent and we are done, or $C$ contains an induced cycle of length $>3$ and we get a contradiction with the chordality of the graph $R$. 
\end{proof}

Let $x^*,y^*$ be a diametral pair of the graph $R$, i.e., $x^*,y^*$ is a pair of vertices of $R$ at maximal distance. Let $P^*$ be any shortest $(x^*,y^*)$-path. 

\begin{claim} \label{c:distance-chordal} Any vertex $z$ of $R$ is at distance at most $N$ from $I(x^*,y^*)$ and thus at distance at most $N+1$ from $P^*$. 
\end{claim}

\begin{proof} If $z\in I(x^*,y^*)$, then we are done. Now, suppose that $z\notin I(x^*,y^*)$.  Consider any quasi-median $x'y'z'$ of the triplet $x^*,y^*,z$. Recall that $x'y'z'$ is a strongly equilateral metric triangle.  Since $R$ is chordal, this metric triangle has side at most 2. Since $R$ does not contain 3-fans, $x'y'z'$ is a metric triangle with side 0 or 1 (otherwise, applying (TC) to $x',y',z'$ we can generate a 3-sun and thus a 3-fan).

If $x'y'z'$ has size 0, then $x'=y'=z':=t$. Since $z\notin I(x^*,y^*)$, necessarily $z\ne t$.
Since $x^*,y^*$ is a diametral pair, $d(z,t)\leq \min\{d(x^*,t), d(y^*,t)\}$.
Hence we can assume that both $d(x^*,t)$ and $d(y^*,t)$ are at least $N+1$ (otherwise we are done).
Pick shortest paths $\rho_{x^*}$, $\rho_{y^*}$ and $\rho_{z}$ from $x^*$, $y^*$ and $z$ to $t$.
Since $t$ is a quasi-median $\rho_{x^*} \cup \rho_{y^*} \cup \rho_{z}$ contains an isometrically embedded $d(z,t)$-tripod.
Thus, $d(z,t)\leq N$ because $R$ does not contain an $N+1$-tripod.

Now suppose that $x'y'z'$ has size 1, i.e., $x',y',z'$ are pairwise adjacent.
If $z=z'$, then $z\sim x',y'\in I(x^*,y^*)$, whence $z$ is located at distance 1 from $I(x^*,y^*)$.
So suppose that $z\ne z'$.
Once again, since $x^*,y^*$ is a diametral pair, $d(z,z')\leq \min\{d(x^*,x'), d(y^*,y')\}$, so we can assume that $d(x^*,x')$ and $d(y^*,y')$ are at least $N+1$.
Pick shortest paths $\rho_{x^*}$, $\rho_{y^*}$ and $\rho_{z}$ from $x^*$ to $x'$, $y^*$ to $y'$ and $z$ to $z'$.
Since $x',y',z'$ is a quasi-median $\rho_{x^*} \cup \rho_{y^*} \cup \rho_{z}$ is an isometrically embedded $d(z,z')$-asteroid.
Thus, $d(z,z')\leq N$ because $R$ does not contain an $N+1$-asteroid.
This shows that any vertex $z$ of $R$ either belongs to $I(x^*,y^*)$ or is at distance at most $N$ from $I(x^*,y^*)$.
The second assertion follows from Claim \ref{c:distance-interval}.  
\end{proof}

From Claim \ref{c:distance-chordal} one can easily deduce that the graph $R$ and the path $P^*$ are $(1,2N+2)$-quasi-isometric via the map $f$, which associate to each vertex $x$ of $R$ a vertex $f(x)$ of $P^*$ closest to $x$.
Indeed, for any vertices $x,y$ of $R$, by the triangle inequality we get \[\vert d(x,y)-d_{P^*}(f(x),f(y))\vert \le d(x,f(x))+d(y,f(y))\le 2N+2.\] Consequently, the level sets of intervals in $X$ are $(1,2N+2)$-quasi-isometric to segments, thus the systolic complex $X$ is fit.
\end{proof}

\begin{remark}
    From the proof of Proposition \ref{prop:equivalent_definition} one can see that if $X$ is $M$-harmless, then it is $\lceil\frac{M}{2}\rceil$-branchless; and that if $X$ is $N$-branchless, then it is $2N+1$-harmless.
\end{remark}

From Proposition~\ref{prop:equivalent_definition} we deduce the following easy condition on links of a systolic complex implying fitness.

\begin{corollary}\label{c:fit_condition}
    If links of vertices of a systolic complex $X$ do not contain induced doubled asteroids 
    then level sets of $X$ do not contain induced tripods or asteroids.
    It follows that the level sets of intervals in $X$ are $(1,4)$-quasi-isometric to segments, in particular, $X$ is fit. 
\end{corollary}

\begin{proof}
    Let $u,v$ be two vertices of $X$, with $d(u,v)=\ell$, and let $1\leq k \leq \ell-1$.
    Suppose $L_k(u,v)$ contains an induced tripod.
    Then by Lemma \ref{lem:tripod-asteroid} both $L_{k-1}(u,v)$ and $L_{k+1}(u,v)$ contain an induced $0$-asteroid.
    Thus, the link of the central vertex of the tripod in $L_k(u,v)$ contains an induced subgraph isomorphic to a graph obtained by gluing two asteroids along their tips, a contradiction.
    A similar argument works if $L_k(u,v)$ contains an induced asteroid.
\end{proof}

\section{Normal clique paths in bridged graphs}\label{sec:normal_clique_paths}

Normal clique paths in systolic complexes have been defined in \cite{JaSw} in order to prove that systolic groups are biautomatic. They are analogues of normal cube-paths in CAT(0) cube complexes.  In \cite{ChChGi}, the approach of \cite{JaSw} was generalized to graphs with convex balls. 
Notice that normal clique paths are usually defined in two different ways, globally and locally, and it is proved that both definitions are equivalent, thus establishing that normal clique paths are canonically and uniquely defined. In this paper, we only need the global definition of normal clique paths, which we present now and closely follow \cite{ChChGi}. 

For a set $K$ of vertices of a bridged graph
$G = (V,E)$ and an integer $k \ge 0$, let
$B^*_k(K)=\bigcap_{s\in K} B_k(s)$.  If $K \subseteq K'$, then $B^*_k(K')\subseteq B^*_k(K)$. If $K$ is a clique, then $B^*_1(K)$ is the union of $K$ and the set of vertices adjacent to all vertices in $K$. 
For a vertex $u$ and  a clique $\tau$ of $G$,  a \emph{clique path} from $u$ to $\tau$ is a
sequence of cliques
$(\sg{u} = \sigma^{(0)}, \sigma^{(1)}, \ldots, \sigma^{(k)}=\tau)$ such that any
two consecutive cliques $\sigma^{(i)}$ and $\sigma^{(i+1)}$ are disjoint and
their union induces a clique of $G$. 
%
Now suppose that the clique $\tau$ is at uniform distance $k$ from the vertex $u$. 
For each $i$ running from $k$ to 0, we inductively define the sets  $\sigma_{(u,\tau)}^{(i)}$ by setting
$\sigma_{(u,\tau)}^{(k)}:=\tau$ and
$\sigma_{(u,\tau)}^{(i)}:=B_1^*\left(\sigma_{(u,\tau)}^{(i+1)}\right)\cap B_{i}(u),$
for any $i\in \sg{0,\ldots, k-1}$. We also let $\sigma_{(u,\tau)}^{(i)}:=\tau$ for any $i \geq k+1$. Let $\gamma_{(u,\tau)}:=(\sg{u} = \sigma_{(u,\tau)}^{(0)}, \sigma_{(u,\tau)}^{(1)},\ldots, \sigma_{(u,\tau)}^{(k-1)},\sigma_{(u,\tau)}^{(k)}=\tau)$. It is shown in \cite{ChChGi} that all sets $\sigma_{(u,\tau)}^{(i)}$ are nonempty cliques of $G$ and that $\gamma_{(u,\tau)}$ is a clique path, which is called the \emph{normal clique path} from the vertex $u$ to the clique $\tau$. If $\tau$ is a single vertex $v$, then the normal clique path from $u$ to $v$ is denoted by $\gamma_{(u,v)}$. Notice that in this case, we additionally have 
the inclusion $\sigma_{(u,v)}^{(i)}\subseteq S_i(u)\cap I(u,v)$. Notice also that the normal clique paths are directed, since the cliques of
$\gamma_{(u,v)}$ and $\gamma_{(v,u)}$ are not the same in general. We will use the following simple 
property of normal clique paths: 

\begin{lemma}
\label{lem: inclalt} If $u$ is a vertex and $\tau',\tau$ are two nonempty cliques at uniform distance $k$ from $u$ such that $\tau'\subseteq \tau$, then  the inclusion $\sigma_{(u,\tau)}^{(k-1)} \subseteq \sigma_{(u,\tau')}^{(k-1)}$ holds.
\end{lemma}

\begin{proof}
Pick  $x\in \sigma_{(u,\tau)}^{(k-1)}$. Then  $x\in S_{k-1}(u)$ and clearly $x$ is adjacent to every vertex of $\tau$, so to every vertex of $\tau'$ as well, which implies that $x \in \sigma_{(u,\tau')}^{(k-1)}$.
\end{proof}

Pick a basepoint $x_0$ and consider the normal clique paths $\gamma_{(x_0,y)}$, $y\in V$. Next we will show that if $x$ is close to $y$, then  $\gamma_{(x_0,y)}$ rapidly converges to the interval $I(x_0,x)$:

\begin{proposition}\label{p:convergence_of_normal_clique_paths} Let $G=(V,E)$ be a bridged graph with a basepoint $x_0$. If $x,y\in V$ with $d(x,y)=k$ and $\gamma_{(x_0,y)}=(\{x_0\} =\sigma_0,\sigma_1,\ldots, \sigma_{m-1},\sigma_m=\{ y\})$
is the normal clique path from $x_0$ to $y$, then there exists an index $i\le 2k$ such that the cliques $\sigma_0,\sigma_1,\ldots, \sigma_{m-i}$ are contained in the interval $I(x_0,x)$.
\end{proposition}

\begin{proof} We proceed by induction on $k=d(x,y)$. We can suppose that $k>0$, otherwise we are done.  Let $x'_0y'x'$ be a quasi-median of the triplet $x_0,y,x$. By Lemma \ref{strongly-equilateral},  $x'_0y'x'$ is a strongly equilateral metric triangle of size $k'\le k$ (thus $d(x',t)=k'$ for any $t\in I(x'_0,y')$).

\begin{case} $y'\ne y$.
\end{case}

Then any neighbour $z$ of $y$ in $I(y,y')\subset I(y,x)$  is  contained in $I(x_0,y)$.
By (INC), $z$ belongs to the before last clique $\sigma_{m-1}$ of the normal clique path $\gamma_{(x_0,y)}$. Also $d(x_0,z)=m-1$ holds.
Consider the normal clique path $\gamma_{(x_0,z)}=(\{x_0\} =\sigma'_0,\sigma'_1,\ldots, \sigma'_{m-2},\sigma'_{m-1}=\{ z\})$.
Since $z\in \sigma_{m-1}$, repeatedly applying Lemma \ref{lem: inclalt}, we will conclude that
$\sigma_{m-2}\subseteq  \sigma'_{m-2}, \sigma'_{m-3}\subseteq \sigma_{m-3},\ldots$. Since $d(z,x)=k-1$, by the induction hypothesis there exists an index $i\le 2(k-1)$ such that $\sigma'_0,\sigma'_1,\ldots, \sigma'_{m-1-i}$
are contained in the interval $I(x_0,x)$. If $\sigma_{m-1-i}\subseteq \sigma'_{m-1-i}\subset I(x_0,x)$, then from the definition of the cliques $\sigma_{m-1-i-1}, \sigma_{m-1-i-2},\ldots, \sigma_1,\sigma_0$, we conclude that
they all are contained in $I(x_0,x)$.  Otherwise, if $\sigma'_{m-1-i}\subseteq \sigma_{m-1-i}$, then $\sigma_{m-1-i-1}\subseteq \sigma'_{m-1-i-1}\subseteq I(x,x_0)$ and again we conclude that all $\sigma_{m-1-i-1},  \sigma_{m-1-i-2},\ldots, \sigma_1, \sigma_0$ are contained in $I(x_0,x)$. Since $i\le 2k-2$, in both cases we conclude that the cliques $\sigma_0,\sigma_1,\ldots,\sigma_{m-2k}$ are contained in $I(x_0,x)$.

\begin{case} $x'\ne x$.
\end{case}

Let $u$ be a neighbour of $x$ in $I(x,x')$. Then $d(y,u)=k-1$ and we can apply the induction hypothesis to $\gamma_{(x_0,y)}$ and $u$ to deduce that the cliques $\sigma_0,\sigma_1,\ldots,\sigma_{m-i}$ are contained in $I(x_0,u)$ for some $i\le 2(k-1)$. Since $I(x_0,u)\subseteq I(x_0,x)$, we also conclude that  $\sigma_0,\sigma_1,\ldots,\sigma_{m-i}$ are contained in $I(x_0,x)$ for $i\le 2(k-1)\le 2k$.

\begin{case} $x'=x$ and $y'=y$.
\end{case}

Then $x'_0yx$ is a  metric triangle of size $k$. Pick any neighbour $z$ of $y$ in $I(y,x'_0)$. By Lemma \ref{strongly-equilateral}, all vertices of $I(x'_0,y)$ (in particular, $z$) have distance $k$ to $x$. This implies that $I(x'_0,z)\cap I(z,x)=\{ z\}$. Since $x'_0yx$ is a metric triangle, we also have $I(x'_0,z)\cap I(x'_0,x)=\{ x'_0\}$. Therefore, any quasi-median of the triplet $x'_0,z,x$ is a metric triangle of size $k-1$ of the form $x'_0zu$, where $u$ is a neighbour of $x$ in $I(x,z)\cap I(x,x'_0)$. 
In particular, we get $d(z,u)=k-1$. 

By (INC),  $z$ belongs to the before last clique $\sigma_{m-1}$ of the normal clique path $\gamma_{(x_0,y)}$. 
Consider the normal clique path $\gamma_{(x_0,z)}=(\{x_0\} =\sigma'_0, \sigma'_1,\ldots, \sigma'_{m-2},\sigma'_{m-1}=\{ z\})$. Since $z\in \sigma_{m-1}$, repeatedly applying Lemma \ref{lem: inclalt}, we will conclude that
$\sigma_{m-2}\subseteq  \sigma'_{m-2}, \sigma'_{m-3}\subseteq \sigma_{m-3},\ldots$. 
Since $d(z,u)=k-1$, by the induction hypothesis, there exists an index $i\le 2(k-1)$ such that the cliques  $\sigma'_0,\sigma'_1,\ldots, \sigma'_{m-1-i}$ are contained in the interval $I(x_0,u)$.
\sloppy If $\sigma_{m-1-i}\subseteq \sigma'_{m-1-i}\subset I(x_0,u)$, then from the definition of the cliques $\sigma_{m-1-i-1}, \sigma_{m-1-i-2} , \ldots , \sigma_1 , \sigma_0$ of $\gamma_{(x_0,y)}$ and since $I(x_0,u)\subset I(x_0,x)$, we conclude that all these cliques are contained in the interval $I(x_0,x)$. 
Since $i\le 2k-2$,  the cliques $\sigma_0,\sigma_1,\ldots,\sigma_{m-2k}$ are contained in $I(x_0,x)$, as required. 
Otherwise, if $\sigma'_{m-1-i}\subseteq \sigma_{m-1-i}$, then $\sigma_{m-1-i-1}\subseteq \sigma'_{m-1-i-1}\subseteq I(x_0,u)\subseteq I(x_0,x)$ and again
we conclude that all the cliques $\sigma_{m-1-i-1}, \sigma_{m-1-i-2},\ldots, \sigma_1,\sigma_0$ are contained in $I(x_0,x)$. Since $i\le 2k-2$, again we conclude that the
cliques $\sigma_0,\sigma_1,\ldots,\sigma_{m-2k}$ are contained in the interval $I(x_0,x)$, concluding the proof.
\end{proof}

\section{The proof of the Main Theorem}\label{sec:Property_A}

We recall a combinatorial criterion for showing that a space has Property A.
This criterion is due to \v{S}pakula and Wright \cite{SpWr}, and is inspired by Brown and Ozawa's \cite{BrOz} proof that hyperbolic groups act amenably on their boundaries.

\begin{proposition}[\cite{SpWr},Proposition 3.1]\label{prop:criterion}
    Let $X$ be a uniformly locally finite, discrete metric space. Suppose that there is an assignment of a nonempty set $S(x,k,\ell)\subseteq X$ for every $\ell\in\mathbb{N}$, $k\in\{1,\ldots,3\ell\}$ and $x\in X$ such that:
    \begin{enumerate}
        \item \label{item:1}for every $\ell\in \mathbb{N}$ there exists $r_\ell>0$ such that each $S(x,k,\ell)$ is included in the ball $B_{r_\ell}(x)$ for all $x\in X$ and $k\in\{1,\ldots,3\ell\}$;
        \item \label{item:2}for every $x,y\in X$, $\ell\geq d(x,y)$ and $k\in\{ \ell+1,\ldots,2\ell\}$ we have inclusions $S(x,k-d(x,y),\ell) \subseteq S(x,k,\ell) \cap S(y,k,\ell)$ and $S(x,k,\ell)\cup S(y,k,\ell) \subseteq S(x,k+d(x,y),\ell)$;
        \item \label{item:3}there exists a function $p$ such that $|S(x,k,\ell)|\leq p(\ell)$ for every $x\in X$, $\ell\in \mathbb{N}$ and $k\in\{1,\ldots,3\ell\}$ with $\lim_{n\to\infty}p(n)^{1/n}=1$.
    \end{enumerate}
\end{proposition}

\begin{proof}[Proof of the Main Theorem]
    Let $X$ be a uniformly locally finite fit systolic complex.
    We need to define sets $S(x,k,\ell)$ as in Proposition \ref{prop:criterion}.
    We begin by fixing a base vertex $x_0\in X$.
    Take a vertex $x\in X$, $\ell\in\mathbb{N}$, and $k\in\{1,\ldots,3\ell\}$.
    For every $y\in B_k(x)$ consider the normal clique path $\gamma_{(x_0,y)}$ from $x_0$ to $y$ and let $\sigma_{(x_0,y)}^{(d(x_0,y)-6\ell)}$ be the clique at distance $6\ell$ from $y$ in $\gamma_{(x_0,y)}$ (or $x_0$ in case $d(x_0,y)<6\ell$).
    We set $S(x,k,\ell)=\bigcup_{y\in B_k(x)}\sigma_{(x_0,y)}^{(d(x_0,y)-6\ell)}$.
    Condition (\ref{item:1}) follows directly from the definition of $S(x,k,\ell)$ and the triangle inequality by setting $r_\ell=9\ell$. To see that condition (\ref{item:2}) is satisfied we will show one inclusion, the other follows similarly. Let us show that $S(x,k-d(x,y),\ell) \subseteq S(x,k,\ell) \cap S(y,k,\ell)$. That $S(x,k-d(x,y),\ell) \subseteq S(x,k,\ell)$ is clear, as $k-d(x,y)<k$, and $S(x,k-d(x,y),\ell) \subseteq S(y,k,\ell)$ because by the triangle inequality we have the inclusion $B_{k-d(x,y)}(x)\subseteq B_{k}(y)$.

    We need to also show that the sets $S(x,k,\ell)$ satisfy condition (\ref{item:3}).
    Since $X$ is a fit systolic complex, there exists constants $B,C>0$ such that each of the level sets
    \[L_{6\ell-k}(x,x_0),L_{6\ell-k+1}(x,x_0), \ldots, L_{6\ell+k}(x,x_0)\]
    is $(B,C)$-quasi-isometric to a segment and has diameter at most $6\ell+k \leq 9\ell$.
    Additionally, the complex $X$ is uniformly locally finite, so the cardinality of the union of these level sets can be quadratically bounded in terms of $\ell$ (note that this bound can be made independent of $x$ and also of $k$, as $k\leq 3\ell$).
    Now, by Proposition \ref{p:convergence_of_normal_clique_paths}, each of the cliques $\sigma_{(x_0,y)}^{(d(x_0,y)-6\ell)}$ is contained in one of the level sets $L_{6\ell-k}(x,x_0),L_{6\ell-k+1}(x,x_0), \ldots, L_{6\ell+k}(x,x_0)$.
    Thus, the cardinality of $S(x,k,\ell)$ can be bounded by a quadratic in $\ell$ function $p$, hence satisfying (\ref{item:3}). This concludes the proof of the theorem. 
\end{proof}

\section{Large-type Artin groups are fit}
\label{sec:Artinfit}

In this section we show that for every almost-large-type Artin group $A_\Gamma$ there exists a uniformly locally finite fit systolic complex $X_\Gamma$  on which $A_\Gamma$ acts geometrically. This proves Corollary A.

We begin by recalling the systolization of the Cayley complex of an almost-large-type Artin group constructed by Huang and Osajda \cite[Definitions 3.7 and 5.3]{HuOs_systolic}. The short description below follows closely the one from \cite[Section 3]{BlVa}.
Let $\tilde{K}_\Gamma$ be the Cayley complex of $A_\Gamma$ corresponding to the standard presentation.
For each edge in $\Gamma$ between vertices $s,t \in V(\Gamma)$ there is a $2$-cell with boundary length $2m_{st}$ that lifts to copies in $\tilde{K}_\Gamma$.
Subdivide each of these lifts to a $2m_{st}$-gon by adding $m_{st}-2$ interior vertices, and call each of these subdivided $2$-cells a precell (see  \cite[Figure 6]{HuOs_systolic} or \cite[Figure 1]{BlVa}).
We call these \emph{new} vertices, and the old vertices \emph{real} vertices.
Now, if two precells $C_1$, $C_2$ intersect at more than one edge, first connect interior vertices $v_1 \in C_1$, $v_2 \in C_2$ such that both $\{v_1\} \cup e$ and $\{v_2\} \cup e$ span triangles for some edge $e \subset C_1 \cap C_2$, and then add edges between interior vertices of $C_1$ and $C_2$ to form a zigzag as in \cite[Figures 7,8,9]{HuOs_systolic} or \cite[Figure 2]{BlVa}.
The flag completion of this complex is systolic, and is denoted by $X_\Gamma$.
It follows from the construction that it is uniformly locally finite.

Now we give a precise description of the links of vertices in $X_\Gamma$.
To do so, we need the following definition:

\begin{figure}[h!]
	\begin{center}
	\includegraphics[scale=0.6]{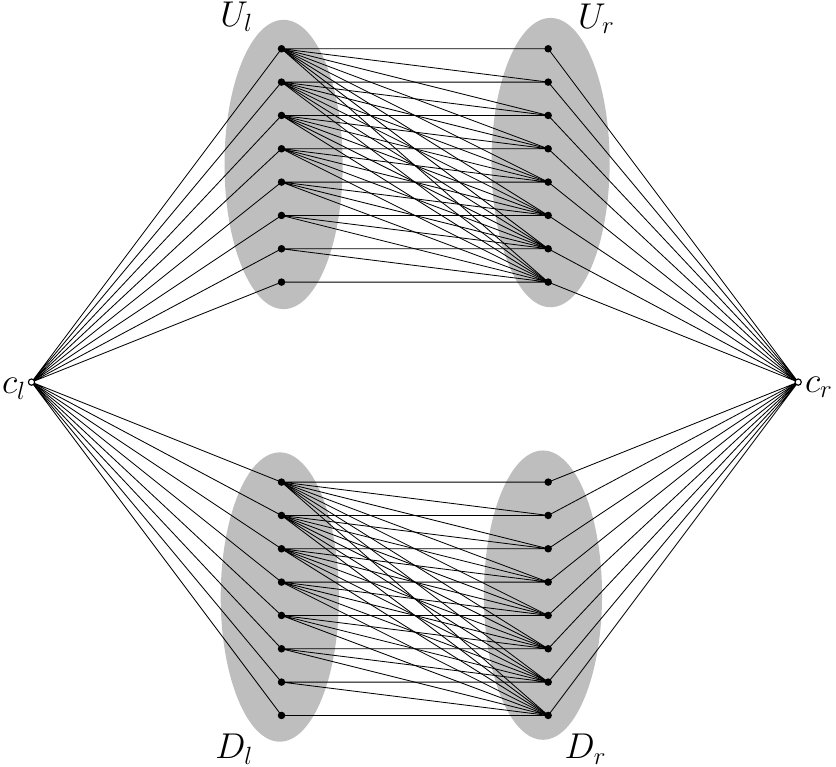}
	\end{center}
	\caption{A thick hexagon. The edges in the complete subgraphs $U_l$, $U_r$, $D_l$ and $D_r$ are not drawn.}
	\label{f:model}
\end{figure}

\begin{definition}
    A \emph{thick hexagon} is a finite simplicial graph such that its vertex set $J$ can be partitioned as $J=\{c_l\} \sqcup \{c_r\} \sqcup U_l \sqcup U_r \sqcup D_l \sqcup D_r$ and the following hold (cf.\ Figure~\ref{f:model}):
    \begin{enumerate}
        \item the vertices adjacent to $c_l$ (respectively, $c_r$) are $U_l \cup D_l$ (respectively, $U_r \cup D_r$);
        \item there are no edges between $U_l\cup U_r$ and $D_l \cup D_r$.
        \item $U_l$, $U_r$, $D_l$ and $D_r$ span cliques of the same cardinality.
        \item it is possible to order the vertices of $U_l$ (respectively, $D_l$) as $\{v_1^{U_l},\ldots,v_k^{U_l}\}$ (respectively, $\{v_1^{D_l},\ldots,v_k^{D_l}\}$) and the vertices of $U_r$ (respectively, $D_r$) as $\{w_1^{U_r},\ldots,w_k^{U_r}\}$ (respectively, $\{w_1^{D_r},\ldots,w_k^{D_r}\}$) so that for each $1 \leq i \leq k$, $v_i^{U_l}$ is adjacent to $w_j^{U_r}$ (respectively, $v_i^{D_l}$ is adjacent to $w_j^{D_r}$) for all $j \geq i$.
    \end{enumerate}
\end{definition}

\begin{proposition} \label{p:new_links} \cite[Subsection 4.3]{HuOs_systolic} The link $\Lk(x)$ of any new vertex $x$ of $X_\Gamma$ is a thick hexagon.
\end{proposition}

\begin{proposition} \label{p:real_links} \cite[Subsection 4.2 and Lemma 5.4]{HuOs_systolic} The link $\Lk(x)$ of any real vertex $x$ of $X_\Gamma$ has the following structure.
For each edge with labels $a,b$ in $\Gamma$ it has a thick hexagon, where
\begin{itemize}
    \item $v_{m_{ab}-1}^{U_l}$ is a real vertex corresponding to an incoming edge labelled by $a$;
    \item $w_1^{U_r}$ is a real vertex corresponding to an outgoing edge labelled by $b$;
    \item $v_{m_{ab}-1}^{D_l}$ is a real vertex corresponding to an incoming edge labelled by $b$;
    \item $w_1^{D_r}$ is a real vertex corresponding to an outgoing edge labelled by $a$;
\end{itemize}
and these thick hexagons are glued as follows. A thick hexagon coming from an edge labelled by $a,b$ is glued to a thick hexagon coming from an edge labelled by $a,c$ at the real vertices coming from incoming and outgoing edges labelled by $a$ listed above; see Figure~\ref{f:gluing}.
\end{proposition}

\begin{figure}[h!]
	\begin{center}
	\includegraphics[scale=0.65]{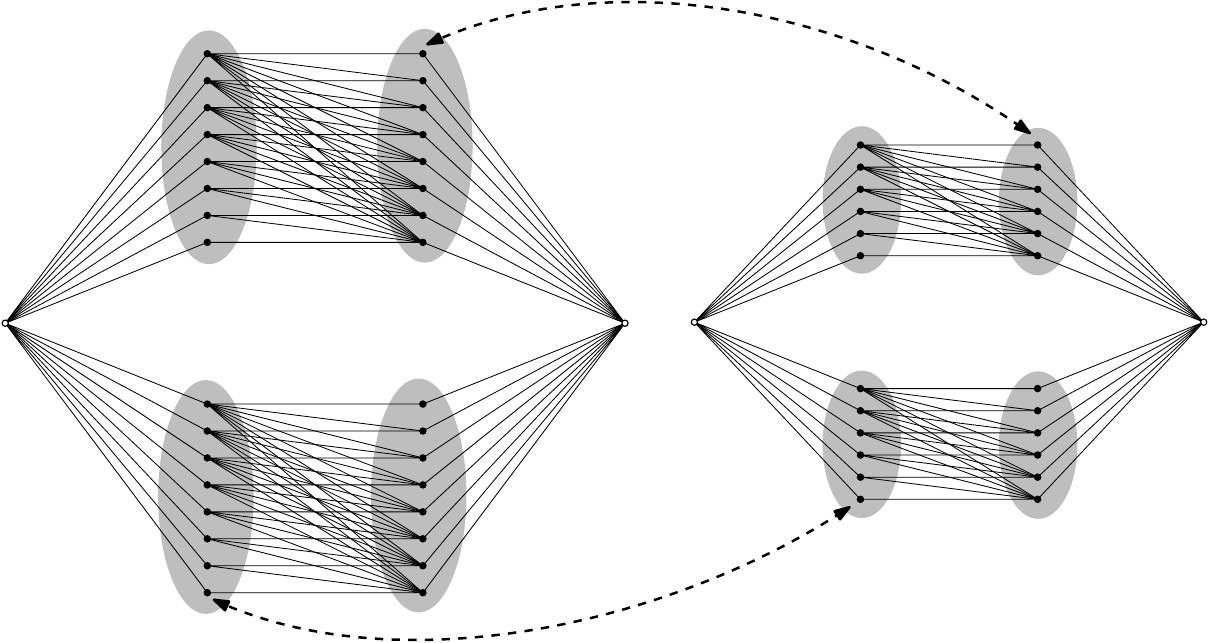}
	\end{center}
	\caption{Gluing of two thick hexagons in the link of a real vertex.}
	\label{f:gluing}
\end{figure}

From the structure of the links and Corollary~\ref{c:fit_condition} we can easily deduce the main result of this section. 

\begin{proposition}\label{prop:fit_Artin} For every almost-large-type Artin group $A_\Gamma$ the level sets in intervals in the complex $X_\Gamma$ do not contain induced tripods and asteroids and hence are $(1,4)$-quasi-isometric to segments. In particular $X_\Gamma$ is fit.
\end{proposition}

\begin{proof}
    By Corollary~\ref{c:fit_condition} it suffices to show that links of vertices in $X_\Gamma$ do not contain induced doubled asteroids as in Figure~\ref{f:triplane} (right).
    Note that triangles in the link of a vertex in $X_\Gamma$ can only appear in the upper ($U_l \sqcup U_r$) or lower ($D_l \sqcup D_r$) part of a thick hexagon.
    It is clear that no two such triangles can be connected by three disjoint paths of length two to form a doubled asteroid.
\end{proof}

\section{Graphical $C(3)$--$T(6)$ groups are fit}\label{sec:c3t6}

In this section we show that for every graphical $C(3)$--$T(6)$ complex $X$ there exists a ``thickening'' $\Th(X)$ which is a fit systolic complex. If $X$ is uniformly locally finite (in the sense explained after Lemma~\ref{lem:edge-pieces}), then $\Th(X)$ is uniformly locally finite. This proves Corollary B.
 
\subsection{Graphical small cancellation complexes}
\label{sec:smallcanc}
We define graphical  small cancellation complexes by closely following \cite[Section 6]{OsPry2018} and \cite[Section 6.4]{CCGHO25}, where the $C(6)$ and $C(4)$--$T(4)$ cases are studied, respectively. In this section we focus on $C(3)$--$T(6)$ complexes. In the following Section~\ref{sec:c6} we deal with $C(6)$ small cancellation complexes.

A map $X\to Y$ between CW complexes is \emph{combinatorial} if its restriction to every open cell of $X$ is a homeomorphism onto an open cell of $Y$.
A CW complex is \emph{combinatorial} if its attaching maps are combinatorial.
In what follows we say $2$-complex to mean combinatorial CW complex of dimension at most $2$.

Let $\Theta$ be a connected simplicial graph without vertices of degree $1$, and let $f\colon \Omega \to \Theta$ be a graph homomorphism. 
Decompose $\Omega = \coprod_i \Omega_i$, where each $\Omega_i$ is a connected component.
There is a natural way of constructing a 2-complex given this information.
First, let $\Theta$ be the 1-skeleton.
For every immersed cycle $C\rightarrow \Omega$, attach a $2$-cell to $\Theta$ along the composition $C\rightarrow \Omega \rightarrow \Theta$.
The resulting complex is called the \emph{graphical complex} associated to $f\colon \Omega \to \Theta$, and we denote it by $X_{(\Omega,\Theta)}$.

A \emph{piece} in $X_{(\Omega,\Theta)}$ is an immersed nontrivial path $\gamma: P\rightarrow \Theta$ which admits two distinct lifts to $\Omega$; that is, there exist two distinct paths $\gamma_1, \gamma_2:P\rightarrow \Omega$ such that $f\circ \gamma_1 = f\circ \gamma_2 = \gamma$ but $\gamma_1\neq \gamma_2$.

A \emph{disc diagram} $\varphi:D\to X_{(\Omega,\Theta)}$ is a combinatorial map, where $D$ is obtained from a combinatorial structure of $S^2$ by removing a $2$-cell.
A \emph{piece} in a disc diagram is a path $P\rightarrow D$ that factors in two distinct ways through 2-cells, i.e., there are 2-cells $R_m\rightarrow D$ and $R_n\rightarrow D$ such that $P\rightarrow D$ factors both as $P\rightarrow R_m \rightarrow D$ and $P\rightarrow R_n \rightarrow D$, and there is no isomorphism $R_m\rightarrow R_n$ making the following diagram commute:
\[
\begin{tikzcd}
    P \arrow[r] \arrow[d] & R_m \arrow[d] \arrow[ld]\\
    R_n \arrow[r] & D
\end{tikzcd}
\]

A disc diagram $D\rightarrow X_{(\Omega,\Theta)}$ is \emph{reduced} if for every piece $P\rightarrow D$ in the disc diagram, $P\rightarrow D\rightarrow X_{(\Omega,\Theta)}$ is again a piece. 

\begin{definition}
    \label{def:smallcanc}
    For $p,q\in \mathbb{N}^+$, we say that the graphical complex $X_{(\Omega,\Theta)}$ associated to $f:\Omega \to \Theta$ satisfies 
    \begin{itemize}
        \item the \emph{C(p) condition} if there is no immersed cycle $C\rightarrow\Omega$ such that $C\rightarrow \Omega \rightarrow \Theta$ is the concatenation of less than $p$ pieces;
        \item the \emph{T(q) condition} if for every reduced disc diagram $D\rightarrow X_{(\Omega,\Theta)}$, there is no interior vertex of $D$ whose degree is greater than $2$ and less than $q$. 
    \end{itemize}  
\end{definition}

The following lemma summarizes some well known properties of graphical $C(3)$--$T(6)$ complexes (see for example \cite{Gui2025}, or \cite[Lemma 6.16]{CCGHO25} for the proofs in the $C(4)$--$T(4)$ case).

\begin{lemma}\label{lem:edge-pieces}
    Let $X_{(\Omega,\Theta)}$ be a simply connected graphical $C(3)$--$T(6)$ complex. Then:
    \begin{itemize}
        \item for every $i$, $\Omega_i\rightarrow \Theta$ is an embedding;
        \item the intersection of any two $\Omega_i$, $\Omega_j$ is empty, a vertex or an edge.
        \item every triangle in $X_{(\Omega,\Theta)}$ bounds a $2$-cell.
    \end{itemize}
\end{lemma}

In general, a graphical complex $X_{(\Omega,\Theta)}$ is not locally finite as a CW-complex. Nevertheless, one can associate to $f:\Omega\rightarrow \Theta$ another complex by attaching, for each component $\Omega_i$, a simplicial cone $C(\Omega_i)$ to $\Theta$ along the restriction $f|_{\Omega_i}$. Denote the resulting complex by $X^*_{(\Omega,\Theta)}$. By Lemma \ref{lem:edge-pieces}, if $X_{(\Omega,\Theta)}$ is simply connected and satisfies the $C(3)$--$T(6)$ condition, then the associated complex $X^*_{(\Omega,\Theta)}$ is a simplicial complex. In this situation, we say that $X_{(\Omega,\Theta)}$ is \emph{uniformly locally finite} if $X^*_{(\Omega,\Theta)}$ is a uniformly locally finite simplicial complex.

\subsection{Systolizing $C(3)$--$T(6)$}

\sloppy Given a simply connected graphical $C(3)$--$T(6)$ complex $X_{(\Omega,\Theta)}$ we construct a systolic complex as follows.
For each $\Omega_i$ add to $\Theta$ every possible edge so that its vertices form a clique, and denote the new simplicial graph by $G_{X_{(\Omega,\Theta)}}$.
The flag completion of $G_{X_{(\Omega,\Theta)}}$ is called the \emph{thickening of $X_{(\Omega,\Theta)}$} and denoted by $\Th(X_{(\Omega,\Theta)})$.
Note that by Lemma \ref{lem:edge-pieces} the vertices of every triangle in $\Th(X_{(\Omega,\Theta)})$ belong to an $\Omega_i$.
Therefore, every maximal simplex of $\Th(X_{(\Omega,\Theta)})$ is precisely the thickening of an $\Omega_i$.
We call such a simplex a \emph{thick cell} in $\Th(X_{(\Omega,\Theta)})$.

\begin{proposition}\label{prop:c3t6_is_systolic}
    Let $X_{(\Omega,\Theta)}$ be a simply connected graphical $C(3)$--$T(6)$ complex, then $\Th(X_{(\Omega,\Theta)})$ is systolic.
\end{proposition}

\begin{proof}
    Clearly $\Th(X_{(\Omega,\Theta)})$ is simply connected if $X_{(\Omega,\Theta)}$ is.
    Thus we have to show that links of vertices do not contain induced $4$-cycles or $5$-cycles.
    The link of a vertex $v$ in the thickened complex $\Th(X_{(\Omega,\Theta)})$ can be described as follows: for each $\Omega_i$ it has a simplex of dimension two less than the number of vertices in $\Omega_i$; and two such simplices share a vertex if and only if the corresponding $\Omega_i$'s intersect at an edge.
    Therefore, induced cycles in $\Lk_{\Th(X_{(\Omega,\Theta)})}(v)$ correspond to induced cycles in $\Lk_{X_{(\Omega,\Theta)}}(v)$, and since $X_{(\Omega,\Theta)}$ is $T(6)$ they have length at least $6$.
\end{proof}

\begin{proposition}\label{prop:c3t6_path_level_sets} Let $X_{(\Omega,\Theta)}$ be a simply connected graphical $C(3)$--$T(6)$ complex. Then for any vertices $u,v\in\Th(X_{(\Omega,\Theta)})$ and $0\le k\le d(u,v)$, the level set $L_k(u,v)$ is a segment.
\end{proposition}

\begin{proof}
    Let $u,v$ be vertices in $\Th(X_{(\Omega,\Theta)})$ at distance $k$.
    We want to show that for every $1<i<k$, the level set $L_i(u,v)$ is a segment.
    By Lemma \ref{lem:edge-pieces} two different thick cells can intersect at most at a single edge.
    Hence level sets do not contain triangles.
    Furthermore, by Proposition \ref{levels-are-chordal} they are convex chordal subgraphs, so they must be trees.
    Suppose there was a tripod in $L_{i+1}(u,v)$, Lemma \ref{lem:tripod-asteroid} we would obtain a triangle in $L_i(u,v)$, a contradiction.
    Consequently, level sets are trees without tripods, i.e. segments.
\end{proof}

\section{Classical $C(6)$ groups are fit}\label{sec:c6}

In this section we prove that simply connected classical $C(6)$ small cancellations complexes are quasi-isometric to fit systolic complexes.
Though unorthodox, classical $C(6)$ small cancellation complexes can be defined as graphical $C(6)$ small cancellation complexes where all the connected components of the corresponding graph $\Omega$ are simplicial circles, and instead of attaching a $2$-cell for each immersed cycle in $\Omega$, we just attach a single $2$-cell along each of the circles in $\Omega$---see Section~\ref{sec:smallcanc} above. The standard reference for the classical small cancellation theory is the book \cite{LynSch2001}.

Let $\mathcal{U}$ be a cover of a set  $X$.
The \emph{nerve} of $\mathcal U$ is the simplicial complex $N({\mathcal{U}})$ whose vertex set is $\mathcal{U}$, and vertices $U_1,\ldots,U_k$ span a $k$-simplex if and only if $\cap_{i=1}^k U_i \neq \varnothing$.
Given a simply connected classical $C(6)$ small cancellation complex $X$, we define its \emph{Wise complex} $W(X)$ as follows. First, we define the complex $X'$ obtained by gluing one hexagonal cell $C$ to every edge $e$ not contained in a $2$-cell of $X$. For any such additional cell $C$ none of its vertices except endpoints of $e$ are contained in $X$ neither in other additional $2$-cells of $X'$. Observe that $X'$ is the union of its $2$-cells and that $X'$ is still a $C(6)$ complex. Now, the Wise complex $W(X)$ of $X$ is the nerve of the cover of $X'$ consisting of the closed $2$-cells of $X'$.
This complex was introduced by Wise in \cite{Wise2003}, where he showed that it is systolic. 
This construction was later generalized to the graphical setting in \cite{OsPry2018}.

We recall the following well-known facts about simply connected $C(6)$ small cancellation complexes. See e.g.\ \cite[Lemma 6.10]{OsPry2018} for a more general statement in the graphical setting.

\begin{proposition}\label{prop:classics}
    Let $X$ be a simply connected classical $C(6)$ small cancellation complex. Then:
    \begin{itemize}
        \item for every $i$, $\Omega_i\rightarrow \Theta$ is an embedding;
        \item the intersection of any two $\Omega_i$, $\Omega_j$ is empty or contractible.
    \end{itemize}
\end{proposition}

The following lemma concerns an important property of links in the Wise complex $W(X)$ of a $C(6)$ complex $X$, and implies immediately fitness of $W(X)$.

\begin{lemma}
\label{lem:linkC(6)}
Let $W(X)$ be the Wise complex of  a classical simply connected $C(6)$ small cancellation complex $X$. Let $x$ be a vertex of $W(X)$ and let $\sigma=(v_0,v_1,\ldots,v_k)$ be a simplex in the link $\Lk(x)$ of $x$ in $W(X)$. Then there exist $i,j\in \{ 0,1,\ldots, k\}$ such that if $v\in \Lk(x)$ is adjacent to $v_l$ for some $l$, then $v\sim v_i$ or $v\sim v_j$. 
\end{lemma}

\begin{proof}
    By Proposition \ref{prop:classics}, the $2$-cells $C_0,\ldots,C_k$ of $X'$ (the augmented complex used to define $W(X)$) corresponding to the vertices of $\sigma$ intersect the $2$-cell $C_x$ corresponding to $x$ in segments of $\partial C_x$, and these segments have nonempty intersection.
    Since $X$ is $C(6)$, these segments cannot cover all of $\partial C_x$.
    Hence, we can order them from left to right, and it is clear that if $C_i\cap C_x$ is one of the left-most and $C_j\cap C_x$ is one of the right-most, then $v_i$ and $v_j$ work.
\end{proof}

\begin{proposition}
    \label{prop:fit_Wise}
    Let $X$ be a simply connected classical $C(6)$ complex. Then its Wise complex $W(X)$ is a fit systolic complex.
\end{proposition}
\begin{proof}
    If $W(X)$ was not fit then, by Corollary~\ref{c:fit_condition} the link of a vertex of $W(X)$ would contain an induced double asteroid, see Figure~\ref{f:triplane} (on the right). This would contradict Lemma~\ref{lem:linkC(6)}.
\end{proof}

Proposition \ref{prop:fit_Wise} does not hold for graphical $C(6)$ small cancellation complexes: 

\begin{example}
Figure~\ref{f:not-fit-C6} shows a part of a graphical $C(6)$ small cancellation complex $X_{(\Omega,\Theta)}$, whose Wise complex $W(X_{(\Omega,\Theta)})$ is not fit. The complex $X_{(\Omega,\Theta)}$ is built out of three ``half-planes''---blue, green, and red---glued together using graphical cells consisting of three paths of length $3$ each sharing endpoints.  

\begin{figure}[h!]
    \includegraphics[scale=0.5]{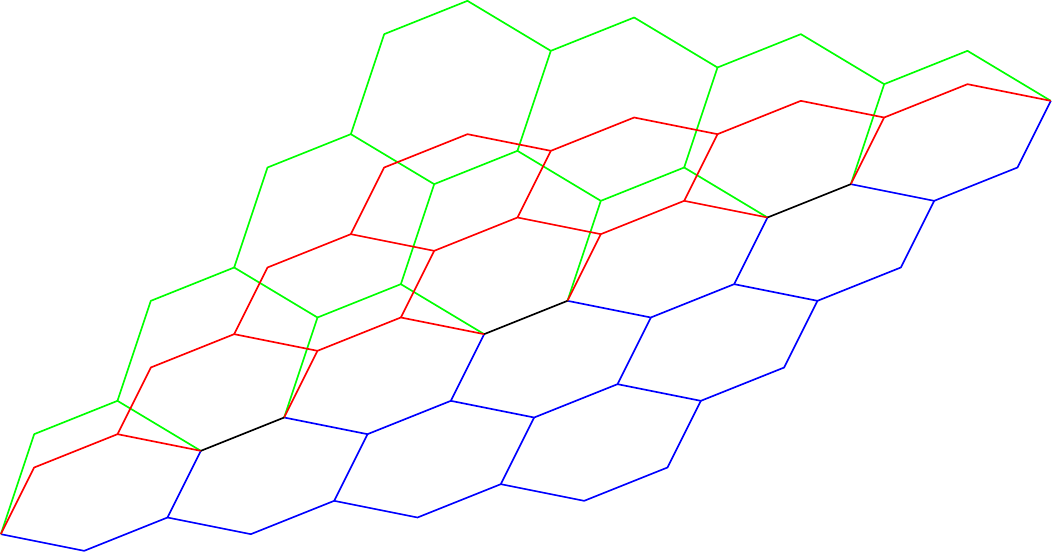}
	\caption{A part of a graphical $C(6)$ complex whose Wise complex is not fit. In fact, the Wise complex of the part shown is a $3$-arrowhead.}
	\label{f:not-fit-C6}
\end{figure}
\end{example}

\section{A three-dimensional example}\label{sec:54547272}

Here we present a fit systolic group acting geometrically
on a $3$-dimensional systolic pseudomanifold.
This group was introduced by J.\ Wieszaczewski in \cite{JW} and follows a construction of a similar group considered first by J.\ \'Swi{\c a}tkowski \cite{Sw}.
Unlike large-type Artin groups and small cancellation groups, this group cannot act geometrically on a $2$-dimensional contractible complex.
In fact, there exists a more general construction of groups of similar type---see Remark~\ref{r:3psmfld} at the end of this section.

Let $T_{54}$ and $T_{72}$ be $6$-large triangulations of the flat $2$-torus consisting of $54$ and $72$ equilateral triangles respectively; see Figure~\ref{f:T54_T72}.
Let $G_{54}$ and $G_{72}$ be the group of automorphisms of, respectively, $T_{54}$ and $T_{72}$ generated by reflections along the edges of triangles.
For the $G_{54}$-action (respectively, $G_{72}$-action) on $T_{54}$ (respectively, $T_{72}$) the stabilisers of triangles are trivial, the stabilisers of edges are isomorphic to $\mathbb Z_2$, and the stabilisers of vertices are isomorphic to the dihedral group $D_3$.
Both actions have as a quotient and fundamental domain a single triangle.

	\begin{figure}[h!]
		\includegraphics[scale=0.25]{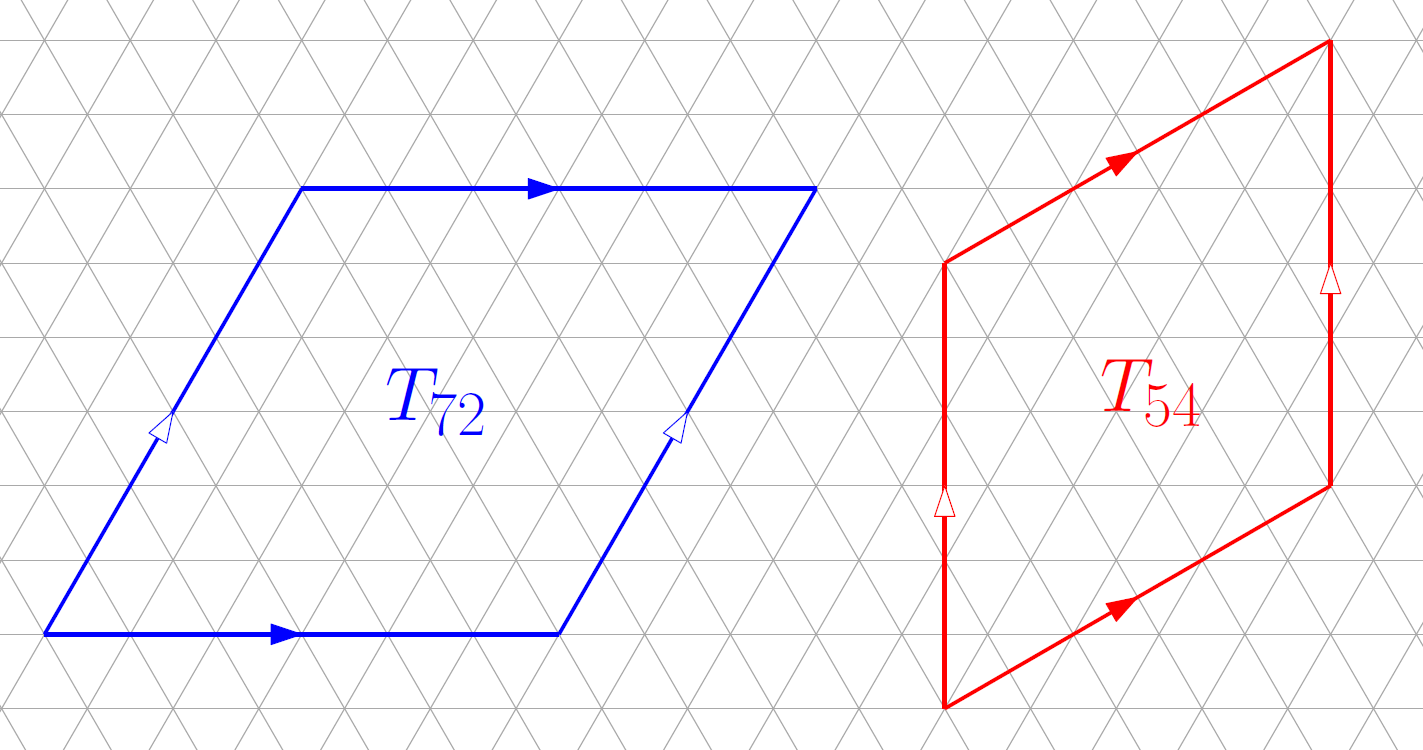}
		\caption{Locally $6$-large tori: $T_{72}$ and $T_{54}$. }
		\label{f:T54_T72}
	\end{figure}	

With this information, one can define a $3$-simplex of groups $\mathcal{G}(54,54,72,72)$ as follows.
The group associated to the $3$-simplex is trivial. The groups associated to the triangles are $\mathbb Z_2$, while the groups associated to the edges are $D_3$. Among the vertex groups, two are $G_{54}$ and the remaining two are $G_{72}$. The inclusion maps correspond to the inclusions
of the respective stabilisers in the $G_{54}$-action on $T_{54}$ (respectively, the $G_{72}$-action on $T_{72}$). See Figure~\ref{f:G72}, and see \cite{JaSw,Sw,JW} for more details.

	\begin{figure}[h!]
		\includegraphics[scale=0.7]{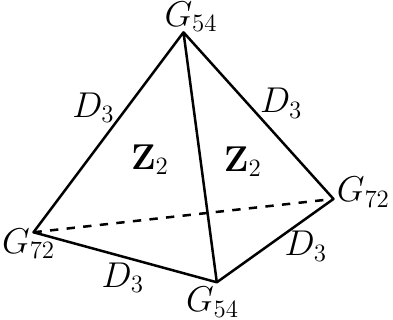}
		\caption{The simplex of groups $\mathcal{G}$.}
		\label{f:G72}
	\end{figure}

Since $\mathcal{G}(54,54,72,72)$ is a locally $6$-large simplex of groups (see \cite{JW}), it is developable by \cite[Theorem 6.1]{JaSw}.
Its fundamental group $G(54,54,72,72)$ acts geometrically (with the corresponding stabilisers of faces) on a development $K(54,54,72,72)$ with each $3$-simplex being a strict fundamental domain (i.e., each $3$-simplex meets each orbit at exactly one point). The complex $K(54,54,72,72)$ is an infinite $3$-dimensional systolic pseudomanifold. For a vertex in $K(54,54,72,72)$, its link is isomorphic to either $T_{54}$ or $T_{72}$, and we say that this vertex is of \emph{type} $54$ or $72$, respectively. Observe that every $3$-simplex has exactly two vertices of type $54$ and two vertices of type $72$. 

Let $E^2_\Delta$ be the equilaterally triangulated Euclidean plane, a \emph{flat} in $K(54,54,72,72)$ is a simplicial map $F\colon E^2_\Delta \rightarrow K(54,54,72,72)$ that restricts to an isometric embedding on the 1-skeleton of $E^2_\Delta$.
For each vertex $v$ in $F$, the intersection $\Lk(v)\cap F$ is a $6$-cycle, we say $v$ is \emph{(non)trivial} in $F$ if $F$ intersects $\Lk(v)$ (non)trivially, i.e., the intersection $\Lk(v)\cap F$ is homotopically (non)trivial in $\Lk(v)$. Proofs of the following elementary Lemmas~\ref{l:uniqueness-Lk(72,nontrivial)}, \ref{l:Lk(trivial)}, and \ref{l:Lk(72,nontrivial)} can be found in \cite{JW}.

\begin{lemma}[\cite{JW}, Lemma 3.8]
\label{l:uniqueness-Lk(72,nontrivial)}
    Up to a simplicial automorphism, there is a unique nontrivial 6-cycle in $T_{72}$. 
\end{lemma}

In particular, for a nontrivial vertex $v$ of type $72$ in a flat $F$, two vertices in $\Lk(v)\cap F$ (a $6$-cycle) belong to the same orbit (of the $G(54,54,72,72)$-action) if and only if they are antipodal. 

\begin{lemma}[\cite{JW}, Lemma 3.14]
\label{l:Lk(trivial)}
    Let $v$ be a trivial vertex  in a flat $F$, then all vertices in $\Lk(v)\cap F$ must be nontrivial in $F$.   
\end{lemma}

\begin{lemma}[\cite{JW}, Lemma 3.15]
\label{l:Lk(72,nontrivial)}
    Let $v$ be a nontrivial vertex of type $72$ in a flat $F$, then all vertices in $\Lk(v)\cap F$ are also nontrivial in $F$. 
\end{lemma}

\begin{proposition}
    $K(54,54,72,72)$ is a uniformly locally finite fit systolic complex.
\end{proposition}

\begin{proof} 

By Proposition \ref{prop:equivalent_definition}, if $K(54,54,72,72)$ was not fit, then it would contain arbitrarily big arrowheads.
Since the group $G(54,54,72,72)$ acts cocompactly on $K(54,54,72,72)$, by König's lemma, $K(54,54,72,72)$
would contain an isometric triplane whose axis is a sequence of $3$-simplices (see Figure \ref{f:triplane}, left).
Then there would be three distinct flats $F_0$, $F_1$, and $F_2$, where any two of them span the entire triplane. Let $\sigma_1=\sigma(v_1,a_1,b_1,c_1)$ and $\sigma'_1=\sigma(v_0,a_1,b_1,c_1)$ be two 3-simplices on the axis that share a triangle, where $a_1\in F_0\cap F_2$, $b_1\in F_0\cap F_1$, and $c_1\in F_1\cap F_2$, see Figure \ref{f:part of a triplane}.

\begin{figure}[h!]
    \centering
    \includegraphics[width=0.5\linewidth]{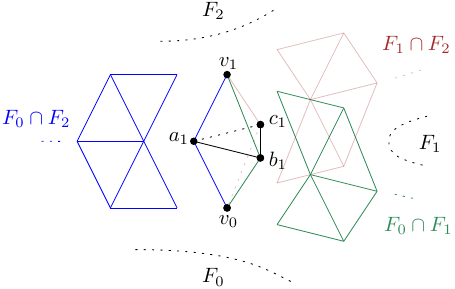}
    \caption{Part of a triplane}
    \label{f:part of a triplane}
\end{figure}

First, observe that $a_1$ cannot be  a nontrivial vertex of type $72$ in $F_0$. Because if it were, by Lemma \ref{l:uniqueness-Lk(72,nontrivial)}, $v_0$ and $v_1$ would lie in different orbits, as they are not antipodal in $\Lk(a_1)\cap F_0$. However, since $\sigma_1$ and $\sigma_1'$ are both strict fundamental domains, $v_0$ and $v_1$ have to be in the same orbit, a contradiction. Second, $a_1$ cannot be a trivial vertex of type $54$ in $F_0$, either. Assume $a_1$ was, then by Lemma \ref{l:Lk(trivial)}, $v_1$ and $b_1$ will be nontrivial in $F_0$, and one of them has to be of type $72$. However, Lemma \ref{l:Lk(72,nontrivial)} then implies $a_1$ is nontrivial in $F_0$, contradicting our assumption.

Therefore, if $a_1$ is of type $72$, then by the previous arguments, $a_1$ has to be trivial in both $F_0$ and $F_2$, which implies, by Lemma \ref{l:Lk(trivial)} and \ref{l:Lk(72,nontrivial)}, that $v_0$, $v_1$, $b_1$, $c_1$ are all of type $54$, since they are adjacent to $a_1$ in either $F_0$ or $F_2$. This is, however, impossible, as $\sigma_1$ should have only two vertices of type $54$. 
Hence, $a_1$ must be of type $54$. Similarly, $b_1$ and $c_1$ also need to be of type $54$, but this again leads to a contradiction. Therefore, such a triplane cannot exist, and $K(54,54,72,72)$ is fit. 
\end{proof}

\begin{remark}
    \label{r:3psmfld}
    In \cite{JW} the author  considers simplices of groups of the form $\mathcal{G}(x,y,w,z)$, with $x,y,w,z\in \{ 2n^2,6n^2  | n\in  \mathbb{N}\}$ ($2n^2$ and $6n^2$ are possible numbers of triangles in triangulations of tori analogous to these in Figure~\ref{f:T54_T72}.) As shown there in some cases their fundamental groups $G(x,y,w,z)$ are hyperbolic, or relatively hyperbolic, or not fit. More precisely, it has been shown that:
   \begin{enumerate}
       \item \label{it:1}
       Flats in $K(x,y,w,z)$ exist only when $x=y=z=w=72$ or $x=y=54$ and $z,w\geq 54$ \cite[Twierdzenie 3.17]{JW};
       \item \label{it:2}$K(72,72,72,72)$ and $K(54,54,w,z)$, for $w,z> 72$ have the \emph{Isolated Flats Property} \cite[Definition 5.1]{Elsner_IF} and \cite[Definicja 4.1]{JW} and \cite[Wniosek 4.5 and 4.7]{JW};
       \item \label{it:3} It follows from \cite[Przykład 3.10]{JW} that $K(54,54,54,54)$ is not fit.
   \end{enumerate}
    Groups $G(x,y,w,z)$ for $K(x,y,w,z)$ different from the ones in (\ref{it:1}) are hyperbolic. Therefore, for them exactness follows from \cite{Ada1994}. 
    From (\ref{it:2}), the groups $G(72,72,72,72)$ and $G(54,54,w,z)$, for $w,z> 72$ are hyperbolic relative to their maximal abelian subgroups. Therefore, for them exactness follows from \cite{Oza2006}.
   
   The group $G(54,54,72,72)$ is not relatively hyperbolic with respect to its maximal abelian subgroups \cite[Wniosek 4.9]{JW}.

   Let us note here that all the general constructions of high-dimensional systolic groups from \cite{JS2003,Hag2003,JaSw,O-chcg} result in hyperbolic groups.
\end{remark}

\section{Final remarks and further questions}\label{sec:final} 

In this article we apply a characterization of Property A by \v{S}pakula and Wright (Proposition~\ref{prop:criterion}) to systolic complexes. Following the approach from \cite{SpWr}, we need to restrict to fit systolic complexes (Definition~\ref{def:fit}) in order to have intervals for which condition (\ref{item:3}) in Proposition~\ref{prop:criterion} is clearly satisfied. As can be easily shown, intervals in general uniformly locally finite (or even, acted geometrically upon a group) systolic complexes can be very large---the volume of level sets may grow exponentially. Therefore, the use of \v{S}pakula and Wright's characterization presented here does not work. Nevertheless, we believe that the answers to the following questions are affirmative.

\begin{question}
    Are uniformly locally finite systolic complexes coarsely embeddable into a Hilbert space? Do they have Yu's Property A?
\end{question}

In this article we provided several classes of groups acting properly on uniformly locally  finite fit systolic complexes---see Sections~\ref{sec:Artinfit}, ~\ref{sec:c3t6},~\ref{sec:c6} \&~\ref{sec:54547272}, and Remark~\ref{rem:2dim}. We are also aware of other examples of groups and spaces coarsely embeddable into such fit systolic complexes. Therefore, we believe that the following problem is important.

\begin{problem}
    Find more groups and spaces coarsely embeddable into uniformly locally finite fit systolic complexes.
\end{problem}

For such groups or spaces Property A (and hence e.g.\ the coarse Baum-Connes conjecture) would follow immediately from the Main Theorem.

In another direction, in this article we show Property A for  $C(3)$--$T(6)$ (Corollary B) and $C(6)$ (Corollary C) small cancellation complexes, whose geometry is tightly related to systolicity.
The remaining $C(4)$--$T(4)$ case is very different.
In that case N. Hoda \cite{Hoda2020} showed that another notion of combinatorial nonpositive curvature---\emph{quadricity} or, in graph terms, \emph{hereditary modularity}---plays an important role. Although there are many analogies between systolic and quadric complexes, at least in the small cancellation case they behave differently: it is easy to produce examples of classical $C(4)$--$T(4)$ complexes---in fact, CAT(0) square complexes being products of two trees---whose quadrization has the property that the size of intervals grows exponentially with their length. To give an illustration, the quadrization of the product of the path of length $k$ with a star consisting of $n$ paths of length $k$ with a common end is the union of $n$ copies of the triangular half of the square $k\times k$ grid glued together along the common diagonal of the grid. The interval in the quadrization between the endpoints of this diagonal is the whole graph, thus it has $O(nk^2)$ vertices. (Notice that in the case  $n=3$, this example can be obtained from the example of Fig. \ref{f:not-fit-C6} of a graphical C(6) complex by contracting all parallel edges forming the smallest positive angle with the positive direction of the abscissa axis). Now, if we will consider the quadrization of the product of a path $P$ of length $k$ with the perfect binary tree $T$ of height $k$, then it has $O(k2^k)$ vertices and they all belong to the interval between the vertices corresponding to the product of the ends of $P$ with the root of $T$. (Notice that the perfect binary tree may also occur in the level sets of intervals of systolic complexes of dimension $3$). Therefore, the techniques developed in this article cannot directly be used in that case.

\begin{question}
    Are uniformly locally finite simply connected $C(4)$--$T(4)$ small cancellation complexes coarsely embeddable into a Hilbert space? Do they have Yu's Property A? Do  uniformly locally finite simply connected graphical C(6) complexes have Property A? 
\end{question}

Finally, it is very interesting whether our proof of boundary amenability could be used to obtain other results, for example in the context of measure equivalence, as e.g.\ in \cite{HoHu2021}.

\subsection*{Acknowledgements} 
We thank Jingyin Huang, Graham Niblo and Piotr Nowak for useful comments.
 The work presented here was partially supported by the Carlsberg Foundation, grant CF23-1226. V.C. was also partially supported by the ANR project MIMETIQUE (ANR-25-CE48-4089-01). D.O.\ was partially supported by (Polish) Narodowe Centrum Nauki, UMO-2018/31/G/ST1/02681.



\bibliographystyle{amsalpha}
\bibliography{mybib}

\end{document}